\documentclass[reqno]{amsart}

\newtheorem{theorem}{Theorem}[section]
\newtheorem{lemma}[theorem]{Lemma}
\newtheorem{proposition}[theorem]{Proposition}

\theoremstyle{definition}
\newtheorem{definition}[theorem]{Definition}

\theoremstyle{remark}
\newtheorem{remark}[theorem]{Remark}

\numberwithin{equation}{section}

\theoremstyle{notation}
\newtheorem{notation}[theorem]{Notation}

\usepackage{graphics}
\usepackage{mathtools}
\usepackage{bbm}
\usepackage[utf8]{inputenc} 
\usepackage[T1]{fontenc}
\usepackage{xcolor}
\usepackage{float}
\usepackage{amssymb}
\usepackage{srcltx}
\usepackage{enumerate}
\usepackage{esint}
\usepackage[shortlabels]{enumitem}
\usepackage{gensymb}

\usepackage[colorlinks,plainpages=true,pdfpagelabels,hypertexnames=true,colorlinks=true,pdfstartview=FitV,linkcolor=blue,citecolor=red,urlcolor=black]{hyperref}

\newtheorem{claim}[theorem]{Claim}

\everymath{\displaystyle}

\newcommand{\abs}[1]{\left\lvert#1\right\rvert}
\newcommand{\norm}[1]{\abs{\abs{#1}}}

\newcommand{\defeq}{\vcentcolon=}
\DeclarePairedDelimiterX{\inp}[2]{\langle}{\rangle}{#1, #2}

\DeclareMathOperator{\st}{{\bf st}}

\newcommand{\starint}{{\prescript{\ast}{}\int}}

\usepackage{graphicx}
\usepackage{amsmath}
\usepackage{tikz}
\usepackage{pgfplots}
\graphicspath{ {images/} }

\usepackage{tikz}
\usepackage{tikz-3dplot}
\usepackage{pgfplots}

\pgfplotsset{%
    compat=1.8,
    compat/show suggested version=false,
}

\usetikzlibrary{calc,fadings,decorations.pathreplacing}
\usepackage[many]{tcolorbox}
\usepackage{wrapfig}

\newcommand\pgfmathsinandcos[3]{%
  \pgfmathsetmacro#1{sin(#3)}%
  \pgfmathsetmacro#2{cos(#3)}%
}
\newcommand\LatitudePlane[3][current plane]{%
  \pgfmathsinandcos\sinEl\cosEl{#2} 
  \pgfmathsinandcos\sint\cost{#3} 
  \pgfmathsetmacro\yshift{\cosEl*\sint}
  \tikzset{#1/.estyle={cm={\cost,0,0,\cost*\sinEl,(0,\yshift)}}} %
}
\newcommand\DrawLatitudeCircle[2][1]{
  \LatitudePlane{\angEl}{#2}
  \tikzset{current plane/.prefix style={scale=#1}}
  \pgfmathsetmacro\sinVis{sin(#2)/cos(#2)*sin(\angEl)/cos(\angEl)}
  \pgfmathsetmacro\angVis{asin(min(1,max(\sinVis,-1)))}
  \draw[current plane] (\angVis:1) arc (\angVis:-\angVis-180:1);
  \draw[current plane,dashed] (180-\angVis:1) arc (180-\angVis:\angVis:1);
}

\begin{document}

\title[Limiting spherical integrals of bounded continuous functions]{Limiting spherical integrals of bounded continuous functions}
\author{Irfan Alam}
\address{Irfan Alam: Department of Mathematics, Louisiana State University, Baton Rouge, LA 70802, USA}
\email{irfanalamisi@gmail.com}
\urladdr{\url{http://www.math.lsu.edu/~ialam1}}


\date{\today}


\begin{abstract}
We use nonstandard analysis to study the limiting behavior of spherical integrals in terms of a Gaussian integral. Peterson and Sengupta proved that if a Gaussian measure $\mu$ has full support on a finite-dimensional Euclidean space, then the expected value of a bounded measurable function on that domain can be expressed as a limit of integrals over spheres $S^{n-1}(\sqrt{n})$ intersected with certain affine subspaces of $\mathbb{R}^n$. This allows one to realize the Gaussian Radon transform of such functions as a limit of spherical integrals. We study such limits in terms of Loeb integrals over a single hyperfinite dimensional sphere. This nonstandard geometric approach generalizes the known limiting result for bounded continuous functions to the case when the Gaussian measure is not necessarily fully supported. We also present an asymptotic linear algebra result needed in the above proof.
\end{abstract}

\maketitle

\begin{section}{Introduction}
The study of the connection between high-dimensional surface area measures on spheres and Gaussian measures, in essence, dates back to the works on the kinetic theory of gas by Boltzmann \cite{Boltzmann} and Maxwell \cite{Maxwell}. Alam \cite{Alam} provided an account of this history, and studied this phenomenon from a nonstandard analytic perspective (see also Cutland--Ng \cite{Cutland-Ng}). A key idea in this work was to express the limiting behavior of spherical integrals through certain Loeb integrals over spheres of hyperfinite dimensions. 

The idea of using a nonstandard measure space as a limiting object of a sequence of measure spaces is applicable to many situations in which we are studying asymptotics of  marginals along a given direction while the ambient spaces are changing. Sengupta \cite{Sengupta} and Peterson--Sengupta \cite{Sengupta-Peterson} studied the Gaussian Radon transform of finite dimensional functions as a limit of spherical integrals over certain spheres of increasing dimension, which is an appropriate setting to work with nonstandard analysis in. This is the main theme of this paper. We refer the reader to \cite{Hertle} and \cite{Hertle2} for earlier standard approaches in this context.

In \cite{Sengupta}, Sengupta fixed a hyperplane $H$ in $\ell^2(\mathbb{R})$ and analyzed the limit of integrals over $S^{n-1}(\sqrt{n})$ intersected with an appropriate ``truncation'' of $H$ to the $n^{\text{th}}$ dimension. More precisely, let $H$ be the set of all square summable real sequences orthogonal to a unit vector $u \in \ell^2(\mathbb{R})$. The integral of a function $f\colon \ell^2(\mathbb{R}) \rightarrow \mathbb{R}$ with respect to the infinite dimensional Gaussian measure with mean $\vec{0} = (0, 0, \ldots)$ and covariance operator equaling the projection $P_H$ onto $H$ is the Gaussian Radon transform of $f$ evaluated at the hyperplane $H$. In general, one could work with a codimension-$1$ affine subspace $A \defeq pu + H$, and integrate $f: \ell^2(\mathbb{R}) \rightarrow \mathbb{R}$ with respect to the Gaussian measure with mean $pu$ and covariance $P_H$ in order to evaluate the Gaussian Radon transform at $A$.

In the case when $f \colon \mathbb{R}^k \to \mathbb{R}$ (identifying it as a function on $\ell^2(\mathbb{R})$ by composing it from the right with the projection to the first $k$ coordinates) is bounded measurable, Sengupta \cite{Sengupta} showed that the corresponding Gaussian Radon transform evaluated at many codimension-$1$ affine subspaces (more precisely, at those affine subspaces for which the marginal onto $\mathbb{R}^k$ of the Gaussian measure described above has full support) can be thought of as limits of spherical integrals of $f$ over the intersection of the spheres $S^{n-1}(\sqrt{n})$ with an appropriate finite dimensional approximation (in $\mathbb{R}^n \subseteq \ell^2(\mathbb{R})$) to $A$. This generalizes the earlier known results on limiting spherical integrals as we are not integrating over the full sphere $S^{n-1}(\sqrt{n})$, but rather on slices of this sphere. In \cite{Sengupta-Peterson}, Peterson and Sengupta generalized the above result further to the case of affine subspaces of any finite codimension.

To more rigorously state the key results in this context, we first need to set up some notation and definitions that will be used throughout the rest of the paper.\\

\subsection{Notation and definitions}
Let $\mathbb{R}^{\mathbb{N}}$ be the vector space of sequences of real numbers, with the standard basis $e_1 = (1, 0, 0, \ldots), e_2 = (0,1,0, \ldots),$ etc. As usual, $\ell^2(\mathbb{R})$ will denote the subspace consisting of all square summable real sequences. For $x = (x_1, x_2, \ldots) \in \mathbb{R}^{\mathbb{N}}$ and $n \in \mathbb{N}$, we define the $n^{\text{th}}$ truncation/projection by
    $$x_{(n)} \defeq (x_1, \ldots, x_n).$$
   
If $x \defeq (x_1, \ldots, x_m) \in \mathbb{R}^m$ for some $m \in \mathbb{N}$, then we will use the same symbol $x$ to denote $(x_1, \ldots, x_m, 0, \ldots, 0) \in \mathbb{R}^n$ for any $n \in \mathbb{N}_{>m}$, as well as to denote $(x_1, \ldots, x_m, 0, 0, \ldots) \in \ell^2(\mathbb{R})$, with the ambient space being clear from the context.     

For $k \in \mathbb{N}$, we use $\pi_{(k)}$ to denote the projection from $\mathbb{R}^{\mathbb{N}}$ (or from some fixed $\mathbb{R}^n$ for $n \in \mathbb{N}_{\geq k}$ if the dimension $n$ is clear from context) onto the first $k$ coordinates: 
\begin{align*}
    \pi_{(k)}(x_1,x_2, \ldots) &= (x_1, \ldots, x_k).
\end{align*}

Let $u^{(1)}, \ldots, u^{(\gamma)}$ be mutually orthonormal vectors in $\ell^2(\mathbb{R})$. For real numbers $p_1, \ldots, p_{\gamma}$ (with $\vec{p} \defeq (p_1, p_2, \ldots, p_{\gamma}) \in \mathbb{R}^{\gamma}$), and $n \in \mathbb{N}$, define (see also Figure \ref{fig:S_A_n_1}):
   \begin{align*}
A(\vec{p}) = A &\defeq \{x \in \ell^2(\mathbb{R}): \langle x, u^{(i)} \rangle = p_i \text{ for all } i \in [\gamma]\},\\
H_n &\defeq \{x \in \mathbb{R}^n: \langle x, (u^{(i)})_{(n)} \rangle = 0 \text{ for all } i \in [\gamma]\}, \\
A_n (\vec{p}) = A_n &\defeq \{x \in \mathbb{R}^n: \langle x, (u^{(i)})_{(n)} \rangle = p_i \text{ for all } i \in [\gamma]\},\\
S_{A_n(\vec{p})} = S_{A_n} &\defeq S^{n-1}(\sqrt{n}) \cap A_n (\vec{p}), \text{ and}\\
S_{H_n} &\defeq S^{n-1}(\sqrt{n}) \cap H_n.
\end{align*}

\begin{centering}
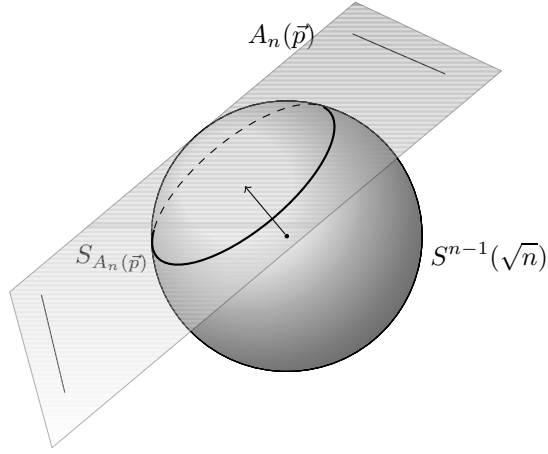
\begin{figure}
    \centering
    \tikzset{
    partial ellipse/.style args={#1:#2:#3}{
        insert path={+ (#1:#3) arc (#1:#2:#3)}
 }}
 {
\centering

\begin{tikzpicture} 
[scale=.6,   rotate= 40]
\def\R{3} 
\def\angEl{40} 
\filldraw[ball color=white] (0,0) circle (\R);
 \foreach \t in {-80,-40,...,80} { \DrawLatitudeCircle[\R]{\t} }

\coordinate [label=left:\textcolor{black}{${}$}](G) at (-6,.5);

\coordinate [label=left:\textcolor{black}{${}$}] (C) at (5,.5);

\coordinate (F) at (4,2.5);

\coordinate (H) at (-5 ,2.5);


 \coordinate [label=left:\textcolor{black}{$A_n(\vec{p})$}] (L) at (3.5, 2.85);
 \coordinate [label=left:\textcolor{black}{$S^{n-1}(\sqrt{n})$}] (X) at
 (4.3, -4.2);
\coordinate [label=left:\textcolor{black}{$S_{A_n(\vec{p})}$}] (L) at (-2.5, 1.45);

 \draw (G)-- (H);

 \draw (C)--(F);


\shadedraw [color=black, opacity = .3]  (-7,-.25)--(6,-.25)--(4.5,3)--(-5.5,3)--(-7,-.25);

 \draw[->, shorten <=1pt,shorten >=1pt](0,0)--(0, 1.5);

\draw[black,fill=black] (0,0) circle (.25ex);

 \draw [dashed] (0,1.5) ellipse (2.5  and 1);

 \draw[thick] (0,1.5) [partial ellipse=150:370: 2.5 and 1];
\end{tikzpicture}
\par}
\label{F:slice}
    \caption{Intersecting $S^{n-1}(\sqrt{n})$ by the affine plane $A_n$}
    \label{fig:S_A_n_1}
\end{figure}
\end{centering}

We also denote $S_{H_n}$ by $S_{n, u^{(1)}, \ldots, u^{(\gamma)}}$ when it is important to emphasize which vectors in $\ell^2(\mathbb{R})$ we are working with. When the sphere is clear from context, we will use $\bar{\sigma}$ to denote its uniform surface area measure. 

The Borel sigma-algebra of a topological space $\Omega$ is denoted by $\mathcal{B}(\Omega)$. Let $\mathcal{S}_0$ be the set of all spheres that are centered at the origin in some real Euclidean space (in any dimension $n \in \mathbb{N}$ and of any radius $r \in \mathbb{R}_{>0}$). For any $S \in \mathcal{S}_0$, we have an orthogonal transformation preserving map called the surface area $\sigma_S \colon \mathcal{B}(S) \rightarrow \mathbb{R}$, which satisfies the following (see \cite[Chapter 3]{spherical_measures} for more background):
\begin{itemize}
\item For any $d \in \mathbb{N}$ and any $a \in \mathbb{R}_{>0}$, we have
        $\sigma_{S^d(a)}(S^d(a)) = c_d \cdot a^{d}$, \text{ where } $c_d = \sigma_{S^d(1)}(S^d(1)) = (d+1) \cdot \frac{\pi^{\frac{d+1}{2}}}{\Gamma\left({\frac{d+1}{2}} + 1\right)} = 2\frac{\pi^{\frac{d+1}{2}}}{\Gamma\left({\frac{d+1}{2}}\right)}$.
\item For any $S \in \mathcal{S}$ and any $A \in \mathcal{B}(\mathcal{S})$, we have $\bar{\sigma}_S(A) = \frac{\sigma_S(A)}{\sigma_S(S)}$.
\end{itemize}

We follow the superstructure approach to nonstandard extensions, as in Albeverio et al. \cite{Albeverio}. In particular, we fix a sufficiently saturated nonstandard extension of a superstructure containing all standard mathematical objects under study. The nonstandard extension of a set $A$ is denoted by ${^*}A$. For $x, y \in {^*}X$ (where $X$ is a normed space), we will write $x \approx y$ to denote that $\norm{x - y}$ is an infinitesimal. The set of finite nonstandard real numbers will be denoted by ${^*}\mathbb{R}_{\text{fin}}$ and the standard part map $\st\colon {^*}\mathbb{R}_{\text{fin}} \rightarrow \mathbb{R}$ takes a finite nonstandard real to its closest real number. We write $N > \mathbb{N}$ (and call such an $N$ \textit{hyperfinite}) if $N \in {^*}\mathbb{N} \backslash \mathbb{N}$. Viewing other spheres as translations of spheres in $S_0$, the concept of the surface area measure canonically extends to all finite-dimensional spheres. By transfer, we have the notion of ${^*}$-surface area in the nonstandard universe. Taking standard parts of the uniform ${^*}$-surface area $\bar{\sigma}_S$ leads to the construction of the uniform Loeb surface measure $L\bar{\sigma}_S$ on any hyperfinite-dimensional sphere $S$. When the sphere is clear from context, we drop the subscript and use $\bar{\sigma}$ and $L\bar{\sigma}$ to denote these measures.

Fix $k \in \mathbb{N}$. For a set $B \in \mathcal{B}(\mathbb{R}^k)$ and any $n \in \mathbb{N}_{\geq k}$, if $S$ is a (possibly lower dimensional) sphere in $\mathbb{R}^n$, then we use $\bar{\sigma}_S(B)$ to denote  $\bar{\sigma}_S\left((B \times \mathbb{R}^{n-k}) \cap S\right)$. Similarly, a function $f\colon \mathbb{R}^k \rightarrow \mathbb{R}$ is canonically extended to $\mathbb{R}^n$ by using $``f(x, y)"$ to denote $f(x)$ for all $x \in \mathbb{R}^k$ and $y \in \mathbb{R}^{n-k}$.

For an element $x$ in a Hilbert space, $P_x$ is the projection operator onto the span of $x$. Let $\bar{\eta} = p_1 (u^{(1)})_{(k)} + \ldots + p_{\gamma} (u^{(\gamma)})_{(k)}$. With $I_k$ being the identity operator on $\mathbb{R}^k$, let ${\mu}_{\bar{\eta}, u^{(1)}, \ldots, u^{(\gamma)}}^{(k)}$ be the Gaussian measure on $\mathbb{R}^k$ with mean $\bar{\eta}$ and covariance 
\begin{align}\label{covariance}
    I_k - {\norm {(u^{(1)})_{(k)}}}^2 P_{(u^{(1)})_{(k)}} - \ldots -  {\norm {(u^{(\gamma)})_{(k)}}}^2 P_{(u^{(\gamma)})_{(k)}}.
\end{align}

We drop the superscript in ${\mu}_{\bar{\eta}, u^{(1)}, \ldots, u^{(\gamma)}}^{(k)}$ when the dimension $k$ is clear from context. Also, when the $u^{(i)}$ and $p_i$ are clear from context, we denote ${\mu}_{\bar{\eta}, u^{(1)}, \ldots, u^{(\gamma)}}$ by just $\mu$. If the $p_i$ are all zero, then we denote the corresponding measure by $\mu_0$.

\subsection{Description of key results}
In the notation set up above, the result of Peterson--Sengupta on limits of integrals on slices of high-dimensional spheres can be summarized as follows (see \cite[Theorem 2.1]{Sengupta-Peterson}).

\begin{theorem}[Peterson--Sengupta]\label{Peterson-Sungupta}
Let $f\colon \mathbb{R}^k \rightarrow \mathbb{R}$ be bounded and Borel measurable. If the Gaussian measure ${\mu}_{\bar{\eta}, u^{(1)}, \ldots, u^{(\gamma)}}$ has full support on $\mathbb{R}^k$, then
\begin{align*}
    \lim_{n \rightarrow \infty} \int_{S_{A_n}} f d{\bar{\sigma}} = \int_{\mathbb{R}^k} f d{\mu}_{\bar{\eta}, u^{(1)}, \ldots, u^{(\gamma)}}.
\end{align*}
\end{theorem}

Nonstandard analysis allows us to view the limit of spherical integrals as (the standard part of) another ``spherical integral'' in a hyperfinite dimension. We will use this idea to generalize Theorem \ref{Peterson-Sungupta} for bounded continuous functions in the case when ${\mu}_{\bar{\eta}, u^{(1)}, \ldots, u^{(\gamma)}}$ does not necessarily have full support on $\mathbb{R}^k$. Our proof is done in several steps of increasing complexity:
 
\begin{enumerate}[(i)]
    \item We first prove Theorem \ref{Peterson-Sungupta} in the case when the coordinates of the vectors $u^{(1)}, \ldots, u^{(\gamma)}$ are zero after a finite index, and the $p_i$ are zero(this is done in the next section---see Lemma \ref{result for eventually zero vectors} and Proposition \ref{basic}).
    
    \item In Section 3, we continue in the case when the $p_i$ are zero (we call this the \textit{case of great circles}), and use overflow to obtain an approximation result for Loeb integrals over a hyperfinite dimensional sphere intersected with an internal affine subspace defined by a hyperfinite truncation of the $u^{(i)}$. Some continuity properties of our integrals then yield the limiting result for bounded uniformly continuous functions in the case of great circles. 
    
    \item We use the scaling and translation properties of the surface area measures to generalize the result for bounded uniformly continuous functions further to the case of non-great circles. See Theorem \ref{main}.
    \item Using Theorem \ref{main}, it follows that almost all points of $S_{A_N}$ (where $N$ is hyperfinite) have finite coordinates along any given direction (this is Theorem \ref{almost all points are finite}). Using this and the notion of $S$-integrability, we are able to finally generalize to all bounded continuous functions (see Theorem \ref{all continuous functions}).
\end{enumerate}

We have tried to keep the nonstandard prerequisites at a minimum. Only a basic familiarity with nonstandard extensions and Loeb measures is needed. While the background provided in Section 2 of \cite{Alam} is sufficient, the interested reader may consult books such as \cite{Albeverio}, \cite{Cutland}, and \cite{Robinson} for further details. We also use a fact from asymptotic linear algebra with a nonstandard proof (Lemma \ref{independence}).
\end{section}

\begin{section}{Integrating bounded functions on certain great circles}
In this section, we prove Theorem \ref{Peterson-Sungupta} in the case when the following hold:
\begin{enumerate}[(i)]
    \item\label{Section 1} The function $f\colon \mathbb{R}^k \rightarrow \mathbb{R}$ is bounded measurable.
    \item\label{Section 2} All the $p_i$ are zero (thus $\bar{\eta}$ is the zero vector in $\mathbb{R}^k$ in this case).
    \item\label{Section 3} The vectors $u^{(i)} \in \ell^2(\mathbb{R})$ are finite-dimensional (their sequence representations with respect to the standard basis have finitely many nonzero terms).
\end{enumerate} 

We will make use of a disintegration formula from \cite{Sengupta}, which we quote below.

\begin{theorem}{\cite[Proposition 4.1, p. 19]{Sengupta}}\label{Theorem of Sengupta}
Let $N$ and $k$ be positive integers with $k < N$, and $f$ any bounded measurable function on $S^{N-1}(a)$, the sphere in $\mathbb{R}^N = \mathbb{R}^k \times \mathbb{R}^{N-k}$ of radius $a$ and with center $0$. Then, with $\sigma$ denoting surface measure (non-normalized) on spheres, we have the following for all $a \in \mathbb{R}_{>0}$: 
\begin{align}\label{Sengupta's formula}
    \int_{z \in S^{N-1}(a)} f(z) d\sigma(z) = \int_{x \in B_k(a)}\left( \int_{y \in S^{N-k-1}(a_x)} f(x, y) d\sigma(y)\right) \frac{a}{a_x} dx,
\end{align}
where $a_x = \sqrt{a^2 - \norm{x}^2}$ and $B_k(a)$ is the open ball of radius $a$ in $\mathbb{R}^k$.
\end{theorem}

The following lemma ensures that the Gaussian measures appearing in the paper are well-defined.

\begin{lemma}\label{consistency} For orthonormal vectors $v^{(1)}, \ldots, v^{(\gamma)}$ in $\ell^2(\mathbb{R})$, the $(k \times k)$ matrix $I - {\norm {(v^{(1)})_{(k)}}}^2 P_{(v^{(1)})_{(k)}} - \ldots -  {\norm {(v^{(\gamma)})_{(k)}}}^2 P_{(v^{(\gamma)})_{(k)}}$ is positive-semidefinite.  In particular, it is the covariance matrix of a Gaussian measure on $\mathbb{R}^k$.
\end{lemma}

\begin{proof}
Let $x \in \mathbb{R}^k$. Then,
\begin{align*}
    &\inp*{x}{\left( I - {\norm {(v^{(1)})_{(k)}}}^2 P_{(v^{(1)})_{(k)}} - \ldots -  {\norm {(v^{(\gamma)})_{(k)}}}^2 P_{(v^{(\gamma)})_{(k)}} \right)x} \\
    = &\norm{x}^2 - \sum_{i = 1}^{\gamma} \inp*{x}{\norm{(v^{(i)})_{(k)}}^2 \inp*{x}{\frac{(v^{(i)})_{(k)}}{\norm{(v^{(i)})_{(k)}}}}{\frac{(v^{(i)})_{(k)}}{\norm{(v^{(i)})_{(k)}}}}} \\
    = &\norm{x}^2 - \sum_{i=1}^{\gamma} \inp*{x}{(v^{(i)})_{(k)}}^2 \\
    \geq &\norm{x}^2 - \sum_{i=1}^{\gamma} \left({\inp*{x}{v^{(i)}}}_{\ell^2}\right)^2 \\
    \geq &\norm{x}^2 - \norm{x}^2 = 0,
\end{align*}
which completes the proof.
\end{proof}

\begin{notation}
The Gaussian measure on $\mathbb{R}^k$ with the above covariance and mean $\rho \in \mathbb{R}^k$ will be denoted by $\mu_{\rho; v^{(1)}, \ldots, v^{(\gamma)}}$. In general, for a positive-semidefinite $(k \times k)$-matrix $L$, we will also use $\mu_{\rho, L}$ to denote the Gaussian  measure on $\mathbb{R}^k$ with mean $\rho$ and covariance $L$.
\end{notation}

We now study the simplest case when the $u^{(i)}$ are all in $\mathbb{R}^k$, the domain of $f$.

\begin{lemma}\label{result for eventually zero vectors}
Let $f\colon \mathbb{R}^k \rightarrow \mathbb{R}$ be a bounded measurable function. Let $u^{(1)}, \ldots, u^{(\gamma)}$ be mutually orthonormal vectors in $\mathbb{R}^k$ (hence, $\gamma \leq k$ necessarily). Then we have:
\begin{align*}
   \lim_{n \rightarrow \infty} \int_{S_{n, u^{(1)}, \ldots, u^{(\gamma)}}} f d{\bar{\sigma}} = \int_{\mathbb{R}^k} f d{\mu_{0;u^{(1)}, \ldots u^{(\gamma)}}}.
\end{align*}
\end{lemma}
\begin{proof}
Without loss of generality, let $\gamma < k$ (if $\gamma = k$, then $u^{(1)}, \ldots, u^{(\gamma)}$ span $\mathbb{R}^k$, and hence the above equality is trivial with both sides being identical to zero).  Let $\{u^{(1)}, \ldots, u^{(\gamma)}, z^{(1)}, \ldots z^{(k - \gamma)}\}$ be an orthonormal basis of $\mathbb{R}^k$. Define $g \colon \mathbb{R}^{k- \gamma} \rightarrow \mathbb{R}$ by $g(y_1, \ldots, y_{k - \gamma}) = f(y_1z^{(1)} + \ldots + y_{k - \gamma}z^{(k - \gamma)})$. 

The map $T \colon S_{n, u^{(1)}, \ldots, u^{(\gamma)}} \rightarrow S^{n - \gamma - 1}(\sqrt{n})$ defined as follows is a measure isomorphism:
$$T\left(\sum_{i = 1}^{k - \gamma} y_iz^{(i)} + \sum_{j = k + 1}^n y_j e_j \right) \defeq \sum_{i = 1}^{n - \gamma} y_ie_i.$$
It thus follows that for any bounded measurable function $f\colon \mathbb{R}^k \rightarrow \mathbb{R}$, we have:
\begin{align*}
      &\int_{S_{n, u^{(1)}, \ldots, u^{(\gamma)}}} f(x_1, \ldots, x_k) d{\bar{\sigma}}(x_1, \ldots, x_n) \\
       = &\int_{S^{n - \gamma -1}(\sqrt{n})} g(y_1, \ldots, y_{k - \gamma}) d{\bar{\sigma}}(y_1, \ldots, y_{n- \gamma}) \\
     = &\frac{1}{\sigma(S^{n - \gamma - 1}(\sqrt{n}))} \int_{S^{n - \gamma -1}(\sqrt{n})} g(y_1, \ldots, y_{k - \gamma}) d{\sigma}(y_1, \ldots, y_{n- \gamma}) \\
     = &\frac{1}{c_{n - \gamma - 1} n^\frac{n - \gamma -1}{2}}  \int\limits_{B_{k - \gamma}(\sqrt{n})}  \frac{g(\vec{y}) n^{\frac{1}{2}}}{(n - {\norm{\vec{y}}}^2)^{\frac{1}{2}}} \cdot \sigma\left(S^{n - \gamma - 1 - (k - \gamma)}\left(\sqrt{n - {\norm{\vec{y}}}^2}\right)\right) d{\lambda}(\vec{y}), 
\end{align*}
where $\vec{y} = (y_1, \ldots, y_{k-\gamma})$. We have used Theorem \ref{Theorem of Sengupta} in the last equality above. Simplifying further, we thus obtain the following:
\begin{align}
     &\int_{S_{n, u^{(1)}, \ldots, u^{(\gamma)}}} f(x_1, \ldots, x_k) d{\bar{\sigma}}(x_1, \ldots, x_n) \nonumber \\
     = &\frac{\left(\frac{2\pi^{\frac{n-k}{2}}}{\Gamma\left({\frac{n-k}{2}}\right)} \right)}{\left(\frac{2\pi^{\frac{n-\gamma}{2}}}{\Gamma\left({\frac{n-\gamma}{2}}\right)} \right)} \cdot \frac{1}{n^\frac{n - \gamma -2}{2}} \cdot \int_{B_{k - \gamma}(\sqrt{n})} g(y_1, \ldots, y_{k - \gamma}) {(n - {\norm{\vec{y}}}^2)^{\frac{n - k - 2}{2}}} d{\lambda}(\vec{y}) \nonumber \\
     = &\frac{1}{{(2\pi)}^{\frac{k- \gamma}{2}}} \cdot \frac{\Gamma \left( \frac{n - k}{2} + \frac{k - \gamma}{2}\right)}{\Gamma \left( \frac{n-k}{2}\right) \cdot {\left( \frac{n - k}{2}\right)}^{\frac{k- \gamma}{2}}}  \nonumber \\
     &~\hspace{1cm} \cdot\int_{B_{k - \gamma}(\sqrt{n})} g(y_1, \ldots, y_{k - \gamma}) \cdot \frac{\left((n - {\norm{\vec{y}}}^2)^{\frac{n - k - 2}{2}}\right) \cdot {\left( n - k\right)}^{\frac{k- \gamma}{2}}}{n^\frac{n - \gamma -2}{2}} 
     d{\lambda}(\vec{y}) \nonumber \\
     = &\frac{a_{n,k}b_{n,k}}{{(2\pi)}^{\frac{k- \gamma}{2}}} \cdot \int_{B_{k - \gamma}(\sqrt{n})} g(y_1, \ldots, y_{k - \gamma}) \cdot {\left( 1 - \frac{{\norm{\vec{y}}}^2}{n}\right)}^{\frac{n - k -2}{2}}
     d{\lambda}(\vec{y}), \label{disintegration}
\end{align}
where $a_{n,k} =  \frac{\Gamma \left( \frac{n - k}{2} + \frac{k - \gamma}{2}\right)}{\Gamma \left( \frac{n-k}{2}\right) \cdot {\left( \frac{n - k}{2}\right)}^{\frac{k- \gamma}{2}}}$ and $b_{n,k} = {\left( 1 - \frac{k}{n} \right)}^{\frac{k - \gamma}{2}}$. Note that $$\lim_{n \rightarrow \infty} a_{n, k} = 1 = \lim_{n \rightarrow \infty} b_{n, k}.$$

Since $f$ is bounded, therefore for large values of $n$, the integrand in \eqref{disintegration} is bounded by $\norm{f}_{\infty} \cdot e^{-\frac{\norm{y}^2}{4}}$ in absolute value, the latter being integrable on $\mathbb{R}^{k - \gamma}$. We thus obtain the following by dominated convergence theorem:
\begin{align}
    & \lim_{n \rightarrow \infty} \int_{S_{n, u^{(1)}, \ldots, u^{(\gamma)}}} f(x_1, \ldots, x_k) d{\bar{\sigma}}(x_1, \ldots, x_n) \nonumber \\
    = &\frac{1}{{(2\pi)}^{\frac{k- \gamma}{2}}} \cdot \int_{{\mathbb{R}}^{k - \gamma}} g(y_1, \ldots, y_{k - \gamma}) \cdot e^{\frac{{- \norm{\vec{y}}}^2}{2}}d{\lambda}(\vec{y}) \nonumber \\
    = &\frac{1}{{(2\pi)}^{\frac{k- \gamma}{2}}} \cdot \int_{{\mathbb{R}}^{k - \gamma}} f(y_1z^{(1)} + \ldots + y_{k - \gamma}z^{(k - \gamma)}) \cdot e^{\frac{{- \norm{\vec{y}}}^2}{2}}d{\lambda}(\vec{y}) \nonumber \\
     = &\frac{1}{{(2\pi)}^{\frac{k}{2}}} \cdot \int_{{\mathbb{R}}^{k}} f(y_1z^{(1)} + \ldots + y_{k - \gamma}z^{(k - \gamma)}) \cdot e^{ - \frac{{y_1}^2 + \ldots + {y_k}^2}{2}} d{\lambda}(\tilde{y}), \label{Tonelli}
\end{align}
where $\tilde{y} = (y_1, \ldots, y_k)$.  \eqref{Tonelli} follows from Fubini's theorem, using the fact that the integral over the last $\gamma$ coordinates is ${(2\pi)}^{\frac{\gamma}{2}}$. Rewriting \eqref{Tonelli}, we have:
\begin{align}
     \lim_{n \rightarrow \infty} \int_{S_{n, u^{(1)}, \ldots, u^{(\gamma)}}} f d{\bar{\sigma}} \nonumber
     &= \frac{1}{{(2\pi)}^{\frac{k}{2}}} \cdot \int_{{\mathbb{R}}^{k}} f(P_{u^{(1)}, \ldots, u^{(\gamma)}} \left( \tilde{y} \right)) \cdot e^{\frac{{- \norm{\tilde{y}}}^2}{2}}d{\lambda}(\tilde{y}) \nonumber\\
     &= \int_{{\mathbb{R}}^k} f d{{\mu}_{0; u^{(1)}, \ldots, u^{(\gamma)}}} \label{last line of eventually zero calculation},
\end{align}
completing the proof.
\end{proof}

The following basic fact about Gaussian measures allows us to strengthen Lemma \ref{result for eventually zero vectors} to the case when the vectors $u^{(1)}, \ldots u^{(\gamma)}$ are vectors in $\ell^2(\mathbb{R})$ that are eventually zero (but not necessarily zero after the $k^{\text{th}}$ coordinate)

\begin{lemma}\label{first claim}
Let $u^{(1)},\ldots, u^{(\gamma)}$ be orthonormal vectors in $\mathbb{R}^m$ where $m \in \mathbb{N}_{>k}$. Let $\mu'$ be the Gaussian measure on $\mathbb{R}^m$ with mean $\mathbf{0} \in \mathbb{R}^m$ and covariance $I - \sum_{i = 1}^{\gamma} P_{u^{(i)}}$. Let ${\mu}_{0; u^{(1)}, \ldots, u^{(\gamma)}}$ be the Gaussian measure on $\mathbb{R}^k$ with mean $0 \in \mathbb{R}^k$ and covariance as in \eqref{covariance}. For any bounded measurable function $f\colon\mathbb{R}^k \rightarrow \mathbb{R}$, we have:
 \begin{align}\label{m to k}
     \int_{{\mathbb{R}}^{m}} f(x_1, \ldots, x_{m}) d{\mu'} = \int_{{\mathbb{R}}^k} f(x_1, \ldots, x_k) d{{\mu}_{0; u^{(1)}, \ldots, u^{(\gamma)}}}.
 \end{align}
 \end{lemma}
 
\begin{proof}
The collection of functions satisfying (\ref{m to k}) is closed under taking $\mathbb{R}$-linear combinations and uniform limits. Hence, it is enough to show that indicator functions of Borel subsets of $\mathbb{R}^k$ satisfy (\ref{m to k}). If $X \sim N (\rho, \Sigma)$ is an $m$-dimensional Gaussian random variable, then $X_{(k)} \sim N({\rho}_{(k)}, \Sigma_{(k,k)})$. Let $\rho = \mathbf{0} \in \mathbb{R}^m$ and $\Sigma$ be the matrix of the operator on $\mathbb{R}^m$ given by $I_m - \sum_{i = 1}^{\gamma} P_{u^{(i)}}$. Note that for any $i \in \{1, \ldots, \gamma\}$ and $j \in \{1, \ldots, k\}$, we have $\inp*{e_j}{u^{(i)}} = \inp*{e_j}{(u^{(i)})_{(k)}}$, which implies
 \begin{align*}
     (P_{u^{(i)}}(e_j))_{(k)} &= (\inp{e_j}{u^{(i)}}u^{(i)})_{(k)} \\
     &= \norm{(u^{(i)})_{(k)}}^2 \inp*{e_j}{\frac{(u^{(i)})_{(k)}}{\norm{(u^{(i)})_{(k)}}}}\frac{(u^{(i)})_{(k)}}{\norm{(u^{(i)})_{(k)}}} \\
     &= \norm{(u^{(i)})_{(k)}}^2P_{(u^{(i)})_{(k)}}(e_j). 
 \end{align*}
 
Thus the operator represented by $\Sigma_{(k,k)}$ is $$I_k - {\norm {(u^{(1)})_{(k)}}}^2 P_{(u^{(1)})_{(k)}} - \ldots -  {\norm {(u^{(\gamma)})_{(k)}}}^2 P_{(u^{(\gamma)})_{(k)}},$$ which completes the proof.
 \end{proof}
 
 \begin{proposition}\label{basic}
    If $u^{(1)}, \ldots, u^{(\gamma)}$ are orthonormal vectors in $\mathbb{R}^m$ and $f\colon \mathbb{R}^k \rightarrow \mathbb{R}$ is a bounded measurable function, then
    \begin{align}\label{R^m limit}
         \lim_{n \rightarrow \infty} \int_{S_{n, u^{(1)}, \ldots, u^{(\gamma)}}} f(x_1, \ldots, x_k) d{\bar{\sigma}}(x_1, \ldots, x_n) 
     &= \int_{{\mathbb{R}}^k} f d{{\mu}_{0; u^{(1)}, \ldots, u^{(\gamma)}}}.
    \end{align}
    \end{proposition}

\begin{proof}
 In the case when $m \leq k$, this follows from Lemma \ref{result for eventually zero vectors}. Now suppose $m > k$. By Lemma \ref{result for eventually zero vectors}, the limit on the left side of (\ref{R^m limit}) is equal to $$\int_{{\mathbb{R}}^{m}} f(x_1, \ldots, x_{m}) d{\mu'},$$ 
 where $\mu'$ is the Gaussian measure on $\mathbb{R}^m$ with mean $0$ and covariance  $I_m - \sum_{i = 1}^{\gamma} P_{u^{(i)}}$. The proof is now completed by Lemma \ref{first claim}.
 \end{proof}
\end{section}

\begin{section}{A hyperfinite approximation and integrating on any great circle}
Throughout this section $N > \mathbb{N}$ will be a hyperfinite number. The goal of this section is to generalize Theorem \ref{Peterson-Sungupta} to the case when the function $f\colon \mathbb{R}^k \rightarrow \mathbb{R}$ is continuous with compact support (we will henceforth write $f \in C_c(\mathbb{R}^k)$), while the $p_i$ are zero (i.e., $S_{A_N}$ is a great circle on $S^{N-1}(\sqrt{N})$), with no restriction on the orthonormal vectors $u^{(1)}, \ldots, u^{(\gamma)} \in \ell^2(\mathbb{R})$. The key steps are as follows:
\begin{enumerate}[(i)]
    \item \label{idea 1} Show that the ${^*}$-integral of ${^*}f$ does not change much when $S^{N-1}(\sqrt{N})$ is intersected by two different internal hyperplanes that are infinitesimally close enough to each other. 
    \item\label{idea 2} In view of \ref{idea 1} and certain continuity properties of the Gaussian integral (with varying covariance), prove some continuity results that show that the integrals in Theorem \ref{Peterson-Sungupta} do not change much when we work with two different hyperfinite truncations of the $u^{(i)}$.
    \item \label{idea 3} Use overflow together with the results of Section 2 to get an approximation result when the vectors $u^{(i)}$ are zero after a small but hyperfinite index $M$. Then, use \ref{idea 2} to complete the proof.
\end{enumerate}

These three ideas will be pursued in the next three subsections respectively. 

\subsection{Effect of an infinitesimal rotation on the ${^*}$-integral of a bounded uniformly continuous function}
\begin{proposition}\label{finite almost everywhere}
Let $N > \mathbb{N}$ and $R \in {^*}\mathbb{R}_{>0}$ be such that $\frac{R}{\sqrt{N}} \in {^*}\mathbb{R}_{\text{fin}}$. Then, almost all points of $S^{N-1}(R)$ have finite coordinates in any given direction, i.e., $$L\bar{\sigma}(\{(x_1, \ldots, x_N) \in S^{N-1}(R): x_i \in {^*}\mathbb{R}_{\text{fin}}\}) = 1 \text{ for all } i \leq N.$$

As a consequence, for any $v^{(1)}, \ldots, v^{(\gamma)} \in {^*}\mathbb{R}^N$, almost all points on the sphere ${S^{N-1}(R) \cap {v^{(1)}}^{\perp} \cap \ldots \cap {v^{(\gamma)}}^{\perp}}$ have finite coordinates along a given direction (unit) vector $w \in {^*}\mathbb{R}^N$. That is, given a unit vector $w \in {^*}\mathbb{R}^N$, almost all points $x$ on this sphere have $\inp*{x}{w} \in {^*}\mathbb{R}_{\text{fin}}$.
\end{proposition}
\begin{proof}
The first half of the proposition was proved in \cite[Corollary 4.3]{Alam} for the case $R = \sqrt{N}$ (i.e. for the sphere $S^{N-1}(\sqrt{N})$). The result for $S^{N-1}(R)$ in general then follows from the (transfer of) scaling property of uniform surface area measures. Indeed, for any $B \in {^*}\mathcal{B}(S^{N-1}(\sqrt{N}))$, we have: 
$$\bar{\sigma}_{S^{N-1}(\sqrt{N})}(B) = \bar{\sigma}_{S^{N-1}(R)}\left( \frac{R}{\sqrt{N}}B \right).$$

Now, let 
$$S' \defeq {S^{N-1}(R) \cap {v^{(1)}}^{\perp} \cap \ldots \cap {v^{(\gamma)}}^{\perp}}.$$

Also, let $L$ be the internal span  $\left({^*}\mathbb{R}v^{(1)} + \ldots + {^*}\mathbb{R}v^{(\gamma)} \right)$ of $v^{(1)}, \ldots, v^{(\gamma)}$ in ${^*}\mathbb{R}^N$. Consider an arbitrary unit vector $w \in {^*}\mathbb{R}^N$. We want to show that almost all points on $S'$ have finite coordinates along $w$, i.e., $\inp*{x}{w} \in {^*}\mathbb{R}_{\text{fin}}$ for $L\bar{\sigma}_{S'}$-almost all $x \in S'$. Let $w'$ and $w''$ be the orthogonal projections of $w$ onto $L$ and its orthogonal complement $L^{\perp}$ (in ${^*}\mathbb{R}^N$) respectively. Since $S' \subseteq L^{\perp}$, we have:
\begin{align*}
    \inp*{x}{w} = \inp*{x}{w'} + \inp*{x}{w''} = \inp*{x}{w''} \text{ for all } x \in S'.
\end{align*}
If $w'' = 0$, then clearly all points of $S$ have the coordinate $0$ along $w$. Otherwise, if $w''$ is not zero, then define $$w^{(1)} \defeq \frac{w''}{\norm{w''}}.$$

Let $c = {^*}\dim(L)$. It is clear that $c \leq \gamma$. Extend $w^{(1)}$ to an orthonormal basis $\{w^{(1)}, \ldots, w^{(N - c)}\}$ of $L^{\perp} = {^*}\mathbb{R}^N \cap {v^{(1)}}^{\perp} \cap \ldots \cap {v^{(\gamma)}}^{\perp}$. Consider the map $\phi\colon {^*}\mathbb{R}^N \cap {v^{(1)}}^{\perp} \cap \ldots \cap {v^{(\gamma)}}^{\perp} \rightarrow {^*}\mathbb{R}^{N-c}$ defined by $$\phi(w^{(i)}) = \mathbf{e}_i \text{ for all } i \in [N-c].$$ 

The map $\phi$ restricted to $S'$ is a measure isomorphism onto $S^{N-1 - c}(R)$. The first half of the proposition (applied to $S^{N-1 - c}(R)$) now completes the proof.
\end{proof}

In the following, we use the concept of Separation Property (SP) defined in Appendix B; see \eqref{defining separation property}. Roughly speaking, a set of vectors satisfy SP if they are linearly independent in a non-infinitesimal way (this is made precise in Appendix B). The hypothesis of Theorem \ref{Rotation Claim} is the same as that of Theorem \ref{Rotation Claim duplicate} (with ${^*}\mathbb{R}^N$ as the ambient internal vector space). By an \textit{infinitesimal vector} in ${^*}\mathbb{R}^N$ (where $N \in {^*}\mathbb{N}$), we mean an element $x \in {^*}\mathbb{R}^N$ such that $\norm{x} \approx 0$.

\begin{theorem}\label{Rotation Claim} Fix $N > \mathbb{N}$. For each $i \in \{1, \ldots, \gamma\}$, let $v^{(i)}, v'^{(i)} \in ({^*}\mathbb{R})^N$ be such that the following conditions hold: 
\begin{enumerate}[(i)]
    \item The collections $\{v^{(1)}, \ldots, v^{(\gamma)}\}$ and $\{v'^{(1)}, \ldots, v'^{(\gamma)}\}$ both satisfy the Separation Property (see \eqref{defining separation property}).  
      \item $\norm{v^{(i)}}, \norm{v^{(i)}} \in {^*}\mathbb{R}_{\text{fin}}.$
    \item $\norm{v^{(i)} - v'^{(i)}} \approx 0.$
\end{enumerate}

Then for any bounded and uniformly continuous $f \colon \mathbb{R}^k \rightarrow \mathbb{R}$, we have:. 
\begin{align*}
\int\limits_{S^{(1)}} \st({^*}f(x)) dL{\bar{\sigma}}(x) = \int\limits_{S^{(2)}} \st({^*}f(x)) dL{\bar{\sigma}}(x),
\end{align*}
where $S^{(1)} = S^{N-1}(\sqrt{N}) \cap {v^{(1)}}^{\perp} \cap \ldots \cap {v^{(\gamma)}}^{\perp}$ and $S^{(2)} = S^{N-1}(\sqrt{N}) \cap {v'^{(1)}}^{\perp} \cap \ldots \cap {v'^{(\gamma)}}^{\perp}$.
\end{theorem}

\begin{proof}
The idea is to first show that $S^{(1)}$ and $S^{(2)}$ are spheres of the same topologial dimension that are infinitesimally apart (to be made precise below). See also Figure \ref{fig:S^1_and_S^2}. Note that the hypotheses on the sets $\{v^{(1)}, \ldots, v^{(\gamma)}\}$ and $\{v'^{(1)}, \ldots, v'^{(\gamma)}\}$ are the same as in Theorem \ref{Rotation Claim duplicate}. Thus, using Theorem \ref{Rotation Claim duplicate}, we obtain orthonormal sets of vectors $\{w^{(1)}, \ldots, w^{(\gamma)}\}$ and $\{z^{(1)}, \ldots, z^{(\gamma)}\}$ such that the following hold:
\begin{enumerate}
    \item \label{pre SP1} For any $i \in  \{1, \ldots, \gamma\}$, we have $\text{span}(v^{(1)}, \ldots, v^{(i)}) = \text{span}(w^{(1)}, \ldots, w^{(i)})$ and $\text{span}(v'^{(1)}, \ldots, v'^{(i)}) = \text{span}(z^{(1)}, \ldots, z^{(i)})$.
    
    \item\label{pre SP2} For all $i \in  \{1, \ldots, \gamma\}$, we have $\norm{w^{(i)} - z^{(i)}} \approx 0$.
\end{enumerate}
Define the internal subspaces $H_1 \defeq \{x \in {^*}\mathbb{R}^N: \langle x, w^{(i)} \rangle = 0 \text{ for all } i \in  \{1, \ldots, \gamma\}\}$ and $H_2 \defeq \{x \in {^*}\mathbb{R}^N: \langle x, z^{(i)} \rangle = 0 \text{ for all } i \in  \{1, \ldots, \gamma\}\}$. Therefore, 
\begin{align*}
S^{(1)} &= {S^{N-1}(\sqrt{N}) \cap {v^{(1)}}^{\perp} \cap \ldots \cap {v^{(\gamma)}}^{\perp}} = S^{N-1}(\sqrt{N}) \cap H_1, \text{ and }\\ 
S^{(2)} &= {S^{N-1}(\sqrt{N}) \cap {v'^{(1)}}^{\perp} \cap \ldots \cap {v'^{(\gamma)}}^{\perp}} = S^{N-1}(\sqrt{N}) \cap H_2. 
\end{align*}

Note that $\text{dim}(H_1 \cap H_2) \geq N - 2\gamma$. Obtain an internally orthonormal set $\{c_1, \ldots, c_{N - 2\gamma}\}$ in $H_1 \cap H_2$. For $i \in  \{1, \ldots, N - 2\gamma\}$, define $w^{(\gamma + i)} = z^{(\gamma + i)} = c_i$. We thus have internally orthonormal sets $\{w^{(1)}, \ldots w^{(N - \gamma)}\}$ and $\{z^{(1)}, \ldots z^{(N - \gamma)}\}$ such that $\norm{w^{(i)} - z^{(i)}} \approx 0$ for all $i \in \{1, \ldots, N - \gamma\}$. Now extend to an internal orthonormal basis $\{w^{(1)}, \ldots, w^{(N - \gamma)}, \ldots, w^{(N)}\}$ of ${^*}\mathbb{R}^N$, and inductively define the following for $i \in  \{1, \ldots, \gamma\}$:
\begin{align*}
   z^{(N - \gamma + 1)} &\defeq \frac{w^{(N - \gamma + 1)} - \sum_{j = 1}^{\gamma}\inp{z^{(j)}}{w^{(N - \gamma + 1)}}z^{(j)}}{\norm{w^{(N - \gamma + 1)} - \sum_{j = 1}^{\gamma}\inp{z^{(j)}}{w^{(N - \gamma + 1)}}z^{(j)}}}, \\
   z^{(N - \gamma + i + 1)} &\defeq \frac{z}{\norm{z}}, \text{ where } 
 \end{align*}
$$z = w^{(N - \gamma + i + 1)} - \sum_{j = 1}^{\gamma}\inp{z^{(j)}}{w^{(N - \gamma + i + 1)}}z^{(j)} - \sum_{l = 1}^{i}\inp{z^{(N - \gamma + l)}}{w^{(N - \gamma + i + 1)}}z^{(N - \gamma + l)}.$$
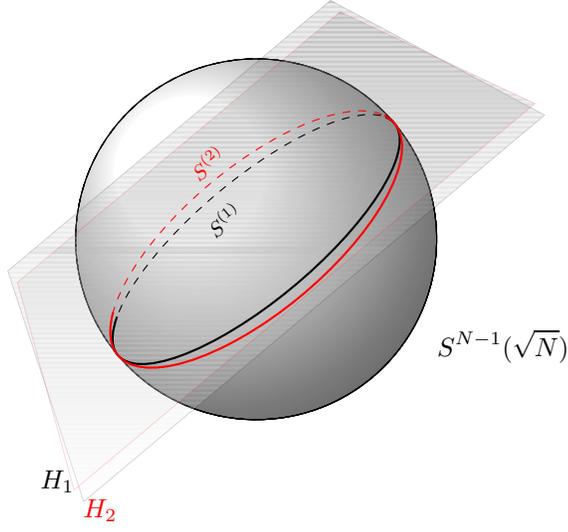
\begin{figure}
    \centering
    \tikzset{
    partial ellipse/.style args={#1:#2:#3}{
        insert path={+ (#1:#3) arc (#1:#2:#3)}
 }}

{\centering
\begin{tikzpicture} 
  [scale=.8,   rotate= 40]

\def\R{3} 
\def\angEl{40} 
\filldraw[ball color=white] (0,0) circle (\R);
\foreach \t in {-80,-40,...,80} { \DrawLatitudeCircle[\R]{\t} }

\coordinate [label=left:\textcolor{black}{${}$}](G) at (-6,.5);
\coordinate [label=left:\textcolor{black}{${}$}] (C) at (5,.5);
\coordinate (F) at (4,2.5);
\coordinate (H) at (-5 ,2.5);


\node at (2,-4) {$S^{N-1}(\sqrt{N})$};

\shadedraw [color=black, opacity = .15]  (-5,-1.5)--(5,-1.5)--(3.5,2.25)--(-3.5,2.25)--(-5,-1.5) node[above left,opacity=1] {$H_1$};
\shadedraw [color=red, opacity = .15]  (-5,-1.25)--(5,-1.25)--(3.5,2)--(-3.5,2)--(-5,-1.25)node[below right,opacity=1] {$H_2$};

\draw [dashed] (0,0) ellipse (3  and 1) node[pos=.5,yshift=.6cm,xshift=-.2cm] {$S^{(1)}$}; 
\draw[thick] (0,0) [partial ellipse=150:370: 3 and 1]; 

\draw [dashed,red] (0,0) ellipse (3  and 1.2) node[pos=.5,yshift=1.5cm,xshift=.2cm] {$S^{(2)}$};; 
\draw[thick,red] (0,0) [partial ellipse=150:370: 3 and 1.2]; 

\end{tikzpicture}
\par}
\label{F:slice_2}
    \caption{$S^{(1)}$ and $S^{(2)}$ are separated infinitesimally}
    \label{fig:S^1_and_S^2}
\end{figure}

It is straightforward to verify that $\{z^{(1)}, \ldots, z^{(N)}\}$ is also an internal orthonormal basis of ${^*}\mathbb{R}^N$ (orthonormality follows by construction and we then use the transfer of the standard fact about Euclidean spaces that says that all orthonormal sets containing as many elements as the dimension span the space). Furthermore, $\norm{w^{(i)} - z^{(i)}} \approx 0$ for all $i \in \{1, \ldots, N\}$ by construction.

Define a ${^*}\mathbb{R}$-linear map $R \colon {^*}\mathbb{R}^N \rightarrow {^*}\mathbb{R}^N$ by $R(w^{(i)}) = z^{(i)}$ for all $i \in \{1, \ldots, N\}$. Since $R$ takes an internal orthonormal basis to an internal orthonormal basis, it is an internal orthogonal map. Also, $R(S^{(1)}) = S^{(2)}$. By transfer, it follows that for any $A \in {^*}\mathcal{B}(S^{(1)})$, we have $\bar{\sigma}_{S^{(2)}}(R(A)) = \bar{\sigma}_{S^{(1)}}(A)$. Hence, it follows that $$L\bar{\sigma}_{S^{(2)}}(R(A)) = L\bar{\sigma}_{S^{(1)}}(A) \text{ for all } A \in {^*}\mathcal{B}(S^{(1)}).$$ 

By a change of variables argument, we conclude the following for any bounded measurable function $g\colon \mathbb{R}^k \rightarrow \mathbb{R}$:
\begin{align*}
    \int_{S^{(2)}} \st({^*}g(x)) dL{\bar{\sigma}}(x) = \int_{S^{(1)}} \st({^*}g(R(x))) dL{\bar{\sigma}}(x).
\end{align*}
Thus it suffices to prove that the following holds for all bounded and uniformly continuous functions $f \colon \mathbb{R}^k \rightarrow \mathbb{R}$. 
\begin{align} \label{requirement}
    \int_{S^{(1)}} \st ({^*}f(x)) dL{\bar{\sigma}}(x) = \int_{S^{(1)}} \st ({^*}f(R(x))) dL{\bar{\sigma}}(x) 
\end{align}

In order to show (\ref{requirement}), we first need the following claim.

\begin{claim}\label{claim infinitesimal rotation}
We have $\norm{x - R(x)} \approx 0$ for almost all $x \in S^{(1)}$.
\end{claim}
\begin{proof}[Proof of Claim \ref{claim infinitesimal rotation}]\phantom\qedhere
Note that $x = \sum_{i = 1}^{N} \inp{x}{w^{(i)}}w^{(i)}$ for any $x \in {^*}\mathbb{R}^N$. Since $w^{(i)} = z^{(i)}$ for all $i \in \{\gamma + 1, \ldots, N- \gamma\}$, the facts that $R(w^{(i)}) = z^{(i)}$ for all $i$ and that $\inp{x}{w^{(i)}} = 0$ for all $i \in \{1, \ldots, \gamma\}$ imply that
\begin{align*}
  \norm{x - R(x)} = \norm{\sum_{i = N - \gamma + 1}^{N} \inp{x}{w^{(i)}}(w^{(i)} - z^{(i)})} \text{ for all } x \in S^{(1)}.
\end{align*}

By Proposition \ref{finite almost everywhere}, it follows that for each $i > \gamma$, almost surely $\inp{x}{w^{(i)}} \in {^*}\mathbb{R}_{\text{fin}}$. Thus, being a maximum of finitely many elements of ${^*}\mathbb{R}_{\text{fin}}$, we have: $$\max\limits_{N - \gamma + 1 \leq i \leq N} \abs{\inp{x}{w^{(i)}}} = s(x) ~(\text{say}),$$ 
where $s(x) \in {^*}\mathbb{R}_{\text{fin}} \text{ for almost all } x \in S^{(1)}$. Hence, for almost all $x \in S^{(1)}$, we get:
\begin{align*}
   \norm{x - R(x)} \leq s(x) \cdot \sum_{i = N - \gamma + 1}^{N} \norm{w^{(i)} - z^{(i)}} \approx 0, \text{ as desired.}
\end{align*}
\end{proof}

Using Claim \ref{claim infinitesimal rotation}, we have $\norm{{^*}\pi_k(x) - {^*}\pi_k(R(x))} \approx 0$ for almost all $x \in S^{(1)}$. Thus, by the nonstandard characterization of uniform continuity, we have that $\st ({^*}f(x)) = \st({^*}f(R(x)))$ for almost all $x \in S^{(1)}$. This completes the proof.
\end{proof}

\subsection{Some integral continuity properties}
\begin{definition}\label{Covariance}
Let $\text{PSD}$ be the set of all positive-semidefinite $(k \times k)$-matrices with real entries. For a bounded measurable function $f\colon \mathbb{R}^k \rightarrow \mathbb{R}$, we define $G_f: \text{PSD} \rightarrow \mathbb{R}$ to be the function that maps $L$ to the expectation of (our fixed function) $f$ with respect to the Gaussian measure on $\mathbb{R}^k$ with mean $0$ and covariance $L$, i.e.,
\begin{align*}
    G_f (L) = \int_{\mathbb{R}^k} f d{\mu_{0, L}}.
\end{align*}
\end{definition}

Being a subset of the space of linear operators on $\mathbb{R}^k$, the space PSD inherits the metric induced by the operator norm. 
\begin{lemma}\label{Covariance continuity}
 With respect to the operator norm on $\text{PSD}$, the map $G_f$ is continuous for all bounded continuous $f \colon \mathbb{R}^k \rightarrow \mathbb{R}$. 
\end{lemma}
\begin{proof}
Let $L_n \rightarrow L$ in PSD. It suffices to prove that $\mu_{0, L_n} \rightarrow \mu_{0, L}$ weakly. Equivalently, we want to show that $Z_n \rightarrow Z$ in distribution, where $Z_n \sim N(0, L_n)$ and $Z \sim N(0, L)$. Since $Z_n$ and $Z$ are $\mathbb{R}^k$-valued Gaussian random variables, $Z_n \rightarrow Z$ in distribution if and only if for any $\vec{x} \in \mathbb{R}^k$, the Gaussian random variables $\inp*{\vec{x}}{Z_n}$ converge in distribution to $\inp*{\vec{x}}{Z}$ (see, for example, \cite[Corollary 4.5, p. 64]{Kallenberg}). Toward that end, fix $\vec{x} \in \mathbb{R}^k$. We know that $\inp*{\vec{x}}{Z_n} \sim N\left(0, \inp*{\vec{x}}{L_n\vec{x}}\right)$, while $\inp*{\vec{x}}{Z} \sim N\left(0, \inp*{\vec{x}}{L\vec{x}}\right)$. For real-valued Gaussian random variables, convergence in distribution is equivalent to the convergence of means and variances. The proof is thus completed by the observation that $L_n\vec{x} \rightarrow L\vec{x}$ (which follows from the fact that $L_n \rightarrow L$ in operator norm).
\end{proof}

\begin{definition}
    For an inner product space $V$ and $\gamma \in \mathbb{N}$, let $V^{[\gamma]}$ be the set of $\gamma$-tuples of orthonormal vectors from $V$.
\end{definition}

\begin{definition}\label{a and theta}
    For a bounded measurable function $f\colon \mathbb{R}^k \rightarrow \mathbb{R}$ and $m \in \mathbb{N}_{\geq k}$, let $\theta_{f,m}\colon (\mathbb{R}^m)^{[\gamma]} \rightarrow \mathbb{R}$ and $a_{f,m}\colon (\mathbb{R}^m)^{\gamma} \times \mathbb{N}_{\geq k} \rightarrow \mathbb{R}$ be defined by 
    \begin{align*}
        \theta_{f, m}(v^{(1)}, \ldots, v^{(\gamma)}) &\defeq \int_{{\mathbb{R}}^k} f d\mu_{0; v^{(1)}, \ldots, v^{(\gamma)}}, \\
        a_{f,m}(v^{(1)}, \ldots, v^{(\gamma)}, n) &\defeq \int_{S_{n, v^{(1)}, \ldots, v^{(\gamma)}}} f d{\bar{\sigma}}.
    \end{align*}
\end{definition}

Here, $S_{n, v^{(1)}, \ldots, v^{(\gamma)}}$ is equal to the intersection of $S^{n-1}(\sqrt{n})$ with $\cap_{i \leq \gamma} (v^{(i)})^\perp$. Note that $\theta_{f, m}$ is defined only on ${(\mathbb{R}^m)}^{[\gamma]}$ (instead of on $(\mathbb{R}^m)^{\gamma}$) since $\mu_{0; v^{(1)}, \ldots, v^{(\gamma)}}$ is well-defined for $(v^{(1)}, \ldots, v^{(\gamma)}) \in {(\mathbb{R}^m)}^{[\gamma]}$ by Lemma \ref{consistency} while it may not be defined in general for an arbitrary set of vectors in $\mathbb{R}^m$. On the other hand, $a_{f,m}(\cdot, n)$ is defined on all of $(\mathbb{R}^m)^{\gamma}$, though we will usually only be interested in the case when $v^{(1)}, \ldots, v^{(\gamma)}$ are truncations of orthonormal vectors in $\ell^2(\mathbb{R})$.

The space $(\mathbb{R}^m)^{[\gamma]}$ inherits a metric from $(\mathbb{R}^m)^{\gamma}$. The next two lemmas respectively prove that under the topology of that metric, the function $\theta_{f, m}$ is continuous and that $a_{f, m}(\cdot, N) \colon{^*}(\mathbb{R}^m)^{[\gamma]} \rightarrow {^*}\mathbb{R}$ is $S$-continuous (an internal function is called $S$-\textit{continuous} if it maps points that are infinitesimally close in the domain to points that are infinitesimally close in the range). 

\begin{lemma}\label{continuity in finite dimensions 1}
     For each bounded continuous function $f\colon \mathbb{R}^k \rightarrow \mathbb{R}$ and $m \in \mathbb{N}$, the map $\theta_{f,m}$ is continuous. In other words, if $(v^{(1)}, \ldots, v^{(\gamma)}) \in (\mathbb{R}^m)^{[\gamma]}$ and $(v'^{(1)}, \ldots, v'^{(\gamma)}) \in ({{^*}{{\mathbb{R}^m}}})^{[\gamma]}$ are such that $\norm{v^{(i)} - v'^{(i)}} \approx 0$ for each $i \in \{1, \ldots, \gamma\}$, then we have  $\theta_{f, m} (v^{(1)}, \ldots, v^{(\gamma)}) \approx {^*}\theta_{f, m}(v'^{(1)}, \ldots, v'^{(\gamma)})$.
\end{lemma}
\begin{proof}
The two statements in the lemma are equivalent by the nonstandard characterization of continuity. We will prove the latter statement. Toward that end, fix $(v^{(1)}, \ldots, v^{(\gamma)}) \in (\mathbb{R}^m)^{[\gamma]}$ and $(v'^{(1)}, \ldots, v'^{(\gamma)}) \in ({{^*}{{\mathbb{R}^m}}})^{[\gamma]}$ such that $$\norm{v^{(i)} - v'^{(i)}} \approx 0 \text{ for each } i \in \{1, \ldots, \gamma\}.$$ 

Let 
\begin{align*}
L &= I - {\norm {(v^{(1)})_{(k)}}}^2 P_{(v^{(1)})_{(k)}} - \ldots -  {\norm {(v^{(\gamma)})_{(k)}}}^2 P_{(v^{(\gamma)})_{(k)}}, \text{ and } \\ L' &= I - {\norm {(v'^{(1)})_{(k)}}}^2 P_{(v'^{(1)})_{(k)}} - \ldots -  {\norm {(v^{(\gamma)})_{(k)}}}^2 P_{(v'^{(\gamma)})_{(k)}}.
\end{align*}

By Definition $\ref{Covariance}$ and transfer, we have $$\theta_{f,m} (v^{(1)}, \ldots, v^{(\gamma)}) = G_f(L), \text{ and } {^*}\theta_{f,m}(v'^{(1)}, \ldots, v'^{(\gamma)}) = {^*}G_f(L').$$ 

By Lemma \ref{Covariance continuity}, the map $G_f$ is continuous. Hence, by nonstandard characterization of continuity, it suffices to show that $\norm{L - L'}_{op} \approx 0$. This is straightforward if one uses the representation of the projection operator as given by the inner product in the direction of the projection.
\end{proof}

\begin{lemma}\label{continuity in finite dimensions 2}
      Let $N > \mathbb{N}$. If $f\colon \mathbb{R}^k \rightarrow \mathbb{R}$ is bounded and uniformly continuous, then the map $a_{f, m}(\cdot, N) \colon{^*}(\mathbb{R}^m)^{[\gamma]} \rightarrow {^*}\mathbb{R}$ is $S$-continuous. 
       
       Equivalently, if $(v^{(1)}, \ldots, v^{(\gamma)}) \text{ and } (v'^{(1)}, \ldots, v'^{(\gamma)}) \in ({{^*}{{\mathbb{R}^m}}})^{[\gamma]}$ are such that $\norm{v^{(i)} - v'^{(i)}} \approx 0$ for each $i \in \{1, \ldots, \gamma\}$, then $${^*}a_{f,m}(v^{(1)}, \ldots, v^{(\gamma)}, N) \approx {^*}a_{f,m}(v'^{(1)}, \ldots, v'^{(\gamma)}, N) \text{ for all } N > \mathbb{N}.$$
\end{lemma}

\begin{proof}
This is immediate from Theorem \ref{Rotation Claim} followed by applications of the transfer principle and the S-integrability of finitely bounded internal functions.      
\end{proof}

\subsection{A hyperfinite approximation via overflow}
\begin{theorem}\label{hyperplane}
With $u^{(1)}, \ldots, u^{(\gamma)}$ orthonormal in $\ell^2(\mathbb{R})$, for any bounded and uniformly continuous $f \colon \mathbb{R}^k \rightarrow \mathbb{R}$, we have
\begin{align*}
    \lim_{n \rightarrow \infty} \int_{S^{n-1}(\sqrt{n}) \cap {u^{(1)}}^{\perp} \cap \ldots \cap {u^{(\gamma)}}^{\perp}} f(x) d{{\bar{\sigma}}}(x) = \int_{\mathbb{R}^k} f(x) d\mu_{0; u^{(1)}, \ldots, u^{(\gamma)}}(x).
\end{align*}
\end{theorem}
\begin{proof}
Fix $f$ as above. Since the limit of a sequence, if it exists, is the same as the standard part of any element with a hyperfinite index in the nonstandard extension of the sequence, it suffices to show that $\st \left({^*}a_{f, N}\left(u^{(1)}_{(N)}, \ldots, u^{(\gamma)}_{(N)}, N\right)\right)$ equals $\int_{\mathbb{R}^k} f(x) d\mu_{0}(x)$. Consider the following internal set:
\begin{align*}
        \mathcal{G} \defeq \Bigg\{m \in {^*}\mathbb{N}: m &\leq N, \text{ and } ~\forall (v^{(1)}, \ldots, v^{(\gamma)}) \in ({{^*}{\mathbb{R}^m}})^{[\gamma]} \\
        &\left(\left| {^*}a_{f, m}(v^{(1)}, \ldots, v^{(\gamma)}, N) - {^*}\theta_{f,m}(v^{(1)}, \ldots, v^{(\gamma)}) \right| < \frac{1}{m} \right)\Bigg\}.
    \end{align*}
    
By Lemma \ref{continuity in finite dimensions 1}, Lemma \ref{continuity in finite dimensions 2} and Corollary \ref{result for eventually zero vectors}, it follows that $\mathbb{N} \subseteq \mathcal{G}$. By overflow, there exists $M > \mathbb{N}$ such that $\{1, \ldots, M\} \subseteq \mathcal{G}$. Fix this $M$. By Lemma \ref{independence}, the vectors $(u^{(1)})_{(M)}, \ldots, (u^{(\gamma)})_{(M)}$ are ${^*}\mathbb{R}$-linearly independent. Use the Gram-Schmidt algorithm to get orthonormal vectors with the same linear span:
\begin{align} \label{Gram-Schmidt}
           w^{(1)} &\defeq \frac{(u^{(1)})_{(M)}}{\norm{(u^{(1)})_{(M)}}}, \\   
         \nonumber  w^{(2)} &\defeq \frac{(u^{(2)})_{(M)} - {\langle (u^{(2)})_{(M)}, w^{(1)}\rangle}w^{(1)}}{\norm{(u^{(2)})_{(M)} - {\langle (u^{(2)})_{(M)}, w^{(1)}\rangle}w^{(1)}}}, \\
         \nonumber w^{(3)} &\defeq \frac{(u^{(3)})_{(M)} - {\langle (u^{(3)})_{(M)}, w^{(1)}\rangle}w^{(1)} - {\langle (u^{(3)})_{(M)}, w^{(2)}\rangle}w^{(2)}}{\norm{(u^{(3)})_{(M)} - {\langle (u^{(3)})_{(M)}, w^{(1)}\rangle}w^{(1)} - {\langle (u^{(3)})_{(M)}, w^{(2)}\rangle}w^{(2)}}} ,\\
         \nonumber  &\vdots \\
         \nonumber  w^{(\gamma)} &\defeq \frac{(u^{(\gamma)})_{(M)} - {\langle (u^{(\gamma)})_{(M)}, w^{(1)}\rangle}w^{(1)} - \ldots - {\langle (u^{(\gamma)})_{(M)}, w^{(\gamma - 1)}\rangle}w^{(\gamma - 1)}}{\norm{(u^{(\gamma)})_{(M)} - {\langle (u^{(\gamma)})_{(M)}, w^{(1)}\rangle}w^{(1)} - \ldots - {\langle (u^{(\gamma)})_{(M)}, w^{(\gamma - 1)}\rangle}w^{(\gamma - 1)}}}.
\end{align}

Since $M \in \mathcal{G}$, we have 
    \begin{align}\label{hyperfinite approximation}
        \left| {^*}a_{f, M}(w^{(1)}, \ldots, w^{(\gamma)}, N) - {^*}\theta_{f, M}(w^{(1)}, \ldots, w^{(\gamma)}) \right| < \frac{1}{M} \approx 0.
    \end{align}

Since $\inp{u^{(i)}}{u^{(j)}}_{\ell^2(\mathbb{R})} = \lim_{m \rightarrow \infty} \inp{(u^{(i)})_{(m)}}{(u^{(j)})_{(m)}} = 0$ if $i \neq j$, the nonstandard characterization of limits implies that 
\begin{align}\label{orthogonal M}
    \inp{(u^{(i)})_{(M)}}{(u^{(j)})_{(M)}} \approx 0 \text{ for } i \neq j.
\end{align}

Similarly, $\norm{(u^{(i)})_{(M)}} \approx \lim_{m \rightarrow \infty} \norm{(u^{(i)})_{(M)}} = 1$, and $\inp{(u^{(i)})_{(M)}}{w^{(j)}} \approx 0$ for all $j \in \{1, \ldots, \gamma\} \backslash \{i\}$ (for a given $i \in \{1, \ldots, j\}$, this follows by induction on $j$ using \eqref{orthogonal M}). Truncating to the first $k$ coordinates, we thus obtain (by induction on $i$):
\begin{align}\label{closeness of k truncations}
    \norm{(w^{(i)})_k - (u^{(i)})_{(k)}} \approx 0 \text{ for all } i.
\end{align} 

Hence the covariance matrix defined by $(w^{(1)})_{(k)}, \ldots, (w^{(\gamma)})_{(k)}$ is infinitesimally close to that defined by $(u^{(1)})_{(k)}, \ldots, (u^{(\gamma)})_{(k)}$ in ${^*}$operator norm, i.e., 
\begin{align}
  &\Bigg\vert \Bigg\vert \left(I - {\norm {(w^{(1)})_{(k)}}}^2 P_{(w^{(1)})_{(k)}} - \ldots -  {\norm {(w^{(\gamma)})_{(k)}}}^2 P_{(w^{(\gamma)})_{(k)}}\right) \nonumber\\
  &- \left(I - {\norm {(u^{(1)})_{(k)}}}^2 P_{(u^{(1)})_{(k)}} - \ldots -  {\norm {(u^{(\gamma)})_{(k)}}}^2 P_{(u^{(\gamma)})_{(k)}}\right) \Bigg\vert \Bigg\vert \approx 0. \label{covariance approximation}
\end{align}

The continuity of $G_f$ thus yields the following: 
\begin{gather*}
    {^*}G_f\left(I - {\norm {(w^{(1)})_{(k)}}}^2 P_{(w^{(1)})_{(k)}} - \ldots -  {\norm {(w^{(\gamma)})_{(k)}}}^2 P_{(w^{(\gamma)})_{(k)}})\right) \approx \int_{\mathbb{R}^k} f d\mu_{0}
\end{gather*}
Also, by transfer we have:
\begin{gather*}
     {^*}G_f\left(I - {\norm {(w^{(1)})_{(k)}}}^2 P_{(w^{(1)})_{(k)}} - \ldots -  {\norm {(w^{(\gamma)})_{(k)}}}^2 P_{(w^{(\gamma)})_{(k)}} \right) \\
     =  {^*}\theta_{f, M}(w^{(1)}, \ldots, w^{(\gamma)}).
\end{gather*}

Hence, using (\ref{hyperfinite approximation}), we get ${^*}a_{f, M}(w^{(1)}, \ldots, w^{(\gamma)}, N) \approx \int_{\mathbb{R}^k} f d\mu_{0}$. Thus, it suffices to show that ${^*}a_{f, M}(w^{(1)}, \ldots, w^{(\gamma)}, N) \approx a_{f, N}((u^{(1)})_{(N)}, \ldots, (u^{(\gamma)})_{(N)}, N)$. Since $f$ is bounded, ${^*}f$ is $S$-integrable on $S^{N-1}(\sqrt{N}) \cap {u^{(1)}}_{(N)}^{\perp} \cap \ldots \cap {u^{(\gamma)}}_{(N)}^{\perp}$, so that the above is equivalent to showing the following for any $f \in C_c(\mathbb{R}^k)$:
\begin{align}
\int_{S^{N-1}(\sqrt{N}) \cap {w^{(1)}}^{\perp} \cap \ldots \cap {w^{(\gamma)}}^{\perp}} &\st({^*}f(x)) dL{\bar{\sigma}}(x) \nonumber\\
\label{claim}
&= \int_{S^{N-1}(\sqrt{N}) \cap {u^{(1)}}_{(N)}^{\perp} \cap \ldots \cap {u^{(\gamma)}}_{(N)}^{\perp}} \st({^*}f(x)) dL{\bar{\sigma}}(x).
\end{align}

This follows from Proposition \ref{hyperfinite truncations are SP} and Theorem \ref{Rotation Claim}, completing the proof.
\end{proof}
\end{section}

\begin{section}{Integrating continuous functions over non-great circles}
In this section, we prove Theorem \ref{Peterson-Sungupta} for all bounded continuous functions. We recall some notation here for convenience. We fix $p_1, \ldots, p_{\gamma} \in \mathbb{R}$, and for any $n \in \mathbb{N}$, we consider the sets 
\begin{align*}
A &\defeq \{x \in \ell^2(\mathbb{R}): \langle x, u^{(i)} \rangle = p_i \text{ for all } i \in \{1, \ldots, \gamma\}\},\\
H_n &\defeq \{x \in \mathbb{R}^n: \langle x, (u^{(i)})_{(n)} \rangle = 0 \text{ for all } i \in \{1, \ldots, \gamma\}\}, \\
A_n &\defeq \{x \in \mathbb{R}^n: \langle x, (u^{(i)})_{(n)} \rangle = p_i \text{ for all } i \in \{1, \ldots, \gamma\}\},\\
S_{A_n} &\defeq S^{n-1}(\sqrt{n}) \cap A_n, \text{ and}\\
S_{H_n} &\defeq S^{n-1}(\sqrt{n}) \cap H_n.
\end{align*}

Let $z^{(1)}, \ldots, z^{(\gamma)}$ be the Gram-Schmidt orthonormalization of the ${^*}\mathbb{R}$-linearly independent vectors $(u^{(1)})_{(N)}, \ldots, (u^{(\gamma)})_{(N)}$ (see Lemma \ref{independence}). Define
  \begin{align*}
   S \defeq S_{H_N} + \left( \frac{p_1}{\norm{(u^{(1)})_{(N)}}}\right )\frac{(u^{(1)})_{(N)}}{{\norm{(u^{(1)})_{(N)}}}} + \ldots + \left( \frac{p_{\gamma}}{\norm{(u^{(\gamma)})_{(N)}}}\right )\frac{(u^{(\gamma)})_{(N)}}{{\norm{(u^{(\gamma)})_{(N)}}}}.
   \end{align*}
   
It is clear that $S_{A_N}$ and $S$ are $(N- \gamma - 1)$-dimensional spheres contained in $A_N$, and that they have the same center $\theta_N$, where

\begin{align}\label{defining theta_N}
   \theta_N &\defeq \left( \frac{p_1}{\norm{(u^{(1)})_{(N)}}}\right )\frac{(u^{(1)})_{(N)}}{{\norm{(u^{(1)})_{(N)}}}} + s\ldots + \left( \frac{p_{\gamma}}{\norm{(u^{(\gamma)})_{(N)}}}\right )\frac{(u^{(\gamma)})_{(N)}}{{\norm{(u^{(\gamma)})_{(N)}}}} \\
   &= q_1z^{(1)} + \ldots  q_{\gamma}z^{(\gamma)} \label{defining theta_N 2} \text{ for some } q_1, \ldots, q_{\gamma} \in {^*}{\mathbb{R}}.
\end{align}

Using the expressions for the $z^{(i)}$ (see \eqref{Gram-Schmidt appendix}) and the fact that $\norm{(u^{(i)})_{(N)}} \approx 1$ for all $i \in \{1, \ldots, \gamma\}$, it follows by induction on $i$ that 
$q_i \approx p_i \text{ for all } i \in \{1, \ldots, \gamma\}$. By truncating onto the first $k$ coordinates in \eqref{defining theta_N} and \eqref{defining theta_N 2}, we thus get:
\begin{align}\label{approximation after truncation}
    q_1(z^{(1)})_{(k)} + \ldots  q_{\gamma}(z^{(\gamma)})_{(k)} \approx  p_1(u^{(1)})_{(k)} + \ldots  p_{\gamma}(u^{(\gamma)})_{(k)}. 
\end{align}

Let $r_N = \frac{\text{Radius}\left( S_{A_N} \right)}{\text{Radius} \left( S \right)} = \frac{\sqrt{N - {q_1}^2 - \ldots - {q_{\gamma}}^2}}{\sqrt{N}} \approx 1$. Then  we have
\begin{align}\label{SLn}
    S_{A_N} = r_N \cdot S_{H_N} + \theta_N.
\end{align}

\begin{centering}
\begin{figure}
    \centering
    \tikzset{
    partial ellipse/.style args={#1:#2:#3}{
        insert path={+ (#1:#3) arc (#1:#2:#3)}
 }}

 {
\centering

   \begin{tikzpicture} 
[scale=.8,   rotate= 40]
\def\R{3} 
\def\angEl{40} 
\filldraw[ball color=white] (0,0) circle (\R);
 \foreach \t in {-80,-40,...,80} { \DrawLatitudeCircle[\R]{\t} }

\coordinate [label=left:\textcolor{black}{${}$}](G) at (-6,.5);

\coordinate [label=left:\textcolor{black}{${}$}] (C) at (5,.5);

\coordinate (F) at (4,2.5);

\coordinate (H) at (-5 ,2.5);


 \coordinate [label=left:\textcolor{black}{$A_N$}] (L) at (3.5, 2.25);
 \coordinate [label=left:\textcolor{black}{$S^{N-1}(\sqrt{N})$}] (X) at
 (3.8, -4.2);
\coordinate [label=left:\textcolor{black}{$\theta_N$}] (L) at (-0, 1.35);
\coordinate [label=left:\textcolor{black}{$\mathbf{0}$}] (L) at (0,0);

\coordinate [label=left:\textcolor{red}{$r_N\sqrt{N}$}] (L) at (1.6, 1.5);
 \coordinate [label=left:\textcolor{black}{$H_N$}] (X) at
 (8, -2);
 
 \draw (G)-- (H);

 \draw (C)--(F);


\shadedraw [color=black, opacity = .3]  (-7,-.25)--(6,-.25)--(4.5,3)--(-5.5,3)--(-7,-.25);
\shadedraw [color=black, opacity = .3]  (-5,-1.25)--(8,-1.25)--(6.5,2)--(-3.5,2)--(-5,-1.25);

 \draw[->, shorten <=1pt](0,0)--(0, 1.5);
 \draw[->, color=red, shorten <=1pt](0,1.5)--(2.5, 1.5);
 
  \draw[->, color=blue, shorten <=1pt](0,0)--(3, 0);
  \coordinate [label=left:\textcolor{blue}{$\sqrt{N}$}] (L) at (1.6, -0.75);

\draw[black,fill=black] (0,0) circle (.25ex)
node[pos=.5,yshift=2.73cm,xshift=-0.25cm] {$S_{A_N}$};

 \draw [dashed] (0,1.5) ellipse (2.5  and 1);

 \draw[thick] (0,1.5) [partial ellipse=150:370: 2.5 and 1];

\draw [dashed] (0,0) ellipse (3  and 1);  
\draw[thick] (0,0) [partial ellipse=150:370: 3 and 1]
node[pos=.5,yshift=0.25cm,xshift=-.2cm] {$S_{H_N}$}; 
\end{tikzpicture}

}
\label{F:slice}
    \caption{Visualizing $S_{A_N}$ in contrast with $S_{H_N}$}
    \label{fig:my_label}
\end{figure}
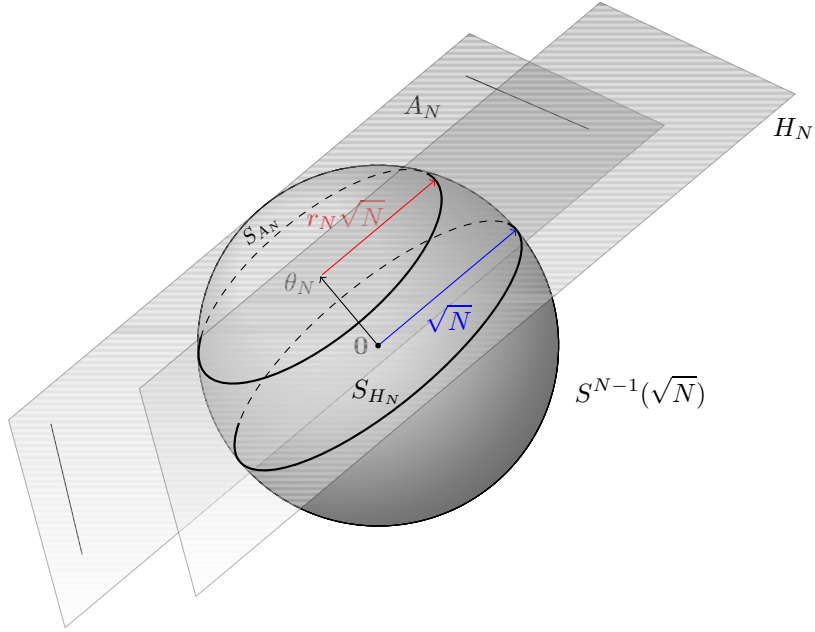\end{centering}

Since the $(u^{(i)})_{(n)}$ are $\mathbb{R}$-linearly independent in $\mathbb{R}^n$ for all large $n \in \mathbb{N}$, we can carry out the above construction to define $\theta_n$ for all $n \in \mathbb{N}_{\geq n'}$, where $n' \in \mathbb{N}$. By the formula corresponding to \eqref{SLn}, we thus have:
\begin{align}\label{formula for calculating measure}
\forall n \in \mathbb{N}_{\geq n'} ~\forall B \in \mathcal{B}(\mathbb{R}^k) \left[\bar{\sigma}_{S_{A_n}}(B) = \bar{\sigma}_{S_{H_n}} \left( \frac{1}{r_n} (B - \pi_k(\theta_n))\right) \right],
\end{align}
where $\pi_k$ denotes the projection onto the first $k$ coordinates under the standard orthonormal basis. We are now in a position to show that $\lim_{n \rightarrow \infty}\int_{S_{A_n}} f d{\bar{\sigma}}$ equals the corresponding Gaussian expectation of $f$ for all $f \in C_c(\mathbb{R}^k)$.

\begin{theorem}\label{main}
Let $f\colon \mathbb{R}^k \rightarrow \mathbb{R}$ be continuous with compact support. Then
\begin{align*}
    \lim_{n \rightarrow \infty} \int_{S_{A_n}} f d{\bar{\sigma}} = \int_{\mathbb{R}^k} f d{\mu}_{\bar{\eta}, u^{(1)}, \ldots, u^{(\gamma)}}.
\end{align*}
\end{theorem}

\begin{proof}
Define $h \colon \mathbb{R}^k \rightarrow \mathbb{R}$ by $h(y) = f\left( y + p_1 (u^{(1)})_{(k)} + \ldots p_{\gamma} (u^{(\gamma)})_{(k)} \right)$ for all $y \in \mathbb{R}^k$. Note the following chain of equations (line 1 follows from transfer of the corresponding expressions for $a_{h,n}$ (as $n$ varies over $\mathbb{N}$) in Definition \ref{a and theta}, line 2 follows from Theorem \ref{hyperplane}, line 3 follows from the definition of $h$, while line 4 follows from properties of Gaussian distributions):
\begin{flalign*}
a_{h, n}(z^{(1)}, \ldots, z^{(\gamma)}, N) &= \starint_{S^{N-1}(\sqrt{N}) \cap {(u^{(1)})_{(N)}}^{\perp} \cap \ldots \cap {(u^{(\gamma)})_{(N)}}^{\perp}} {^*}h d{^*}\bar\sigma &&\\
&\approx \int_{\mathbb{R}^k} h d\mu_{0, u^{(1)}, \ldots, u^{(\gamma)}} &&\\
&= \int_{\mathbb{R}^k} f\left( y + p_1 (u^{(1)})_{(k)} + \ldots p_{\gamma} (u^{(\gamma)})_{(k)} \right) d{\mu}_{0, u^{(1)}, \ldots, u^{(\gamma)}} &&\\
&= \int_{\mathbb{R}^k} f d{\mu}_{\bar{\eta}, u^{(1)}, \ldots, u^{(\gamma)}}, &&
\end{flalign*}
where $\bar{\eta} = p_1 (u^{(1)})_{(k)} + \ldots p_{\gamma} (u^{(\gamma)})_{(k)}$.

The $S$-integrability of ${^*}h$ thus implies that 
\begin{align}\label{useful equation}
    \int_{S^{N-1}(\sqrt{N}) \cap {z^{(1)}}^{\perp} \cap \ldots \cap {z^{(\gamma)}}^{\perp}} \st ({^*}h) dL{\bar{\sigma}} &= \int_{\mathbb{R}^k} f d{\mu}_{\bar{\eta}, u^{(1)}, \ldots, u^{(\gamma)}}.
\end{align}

Using \eqref{useful equation} and the nonstandard characterization of uniform continuity (which, in particular, implies that ${^*}h(x) \approx {^*}h(rx)$ for all $x \in {^*}\mathbb{R}^N$ and $r \approx 1$), we obtain:

\begin{align}
    \int_{\mathbb{R}^k} f d{\mu}_{\bar{\eta}, u^{(1)}, \ldots, u^{(\gamma)}} &= \int_{S^{N-1}(\sqrt{N}) \cap {z^{(1)}}^{\perp} \cap \ldots \cap {z^{(\gamma)}}^{\perp}} \st ({^*}h(x)) dL{\bar{\sigma}}(x) \nonumber \\
    &= \int_{S_{N, z^{(1)}, \ldots, z^{(\gamma)}}} \st({^*}h(rx)) d{L\bar{\sigma}}(x). \label{first break}
\end{align}

Note that the composition of ${^*}h$ with the scaling by $r$ is a finitely bounded, and hence $S$-integrable, function. This and transfer of the scaling properties of the uniform surface measures respectively imply the following: 
\begin{align}\nonumber
    \int_{\mathbb{R}^k} f d{\mu}_{\bar{\eta}, u^{(1)}, \ldots, u^{(\gamma)}} &\approx \starint_{S_{N, z^{(1)}, \ldots, z^{(\gamma)}}} {^*}h(rx) d{^*}\bar{\sigma} \\
    &\approx \starint_{r \cdot S_{N, z^{(1)}, \ldots, z^{(\gamma)}}} {^*}h(x) d{{^*}{\bar{\sigma}}}(x) \label{second break}.
\end{align}

The proof is now contained in the following sequence of equations obtained by simplifying \eqref{second break}:

\begin{align*}
    \int_{\mathbb{R}^k} f d{\mu}_{\bar{\eta}, u^{(1)}, \ldots, u^{(\gamma)}} 
    &\approx \starint_{r \cdot S_{N, z^{(1)}, \ldots, z^{(\gamma)}}} {^*}f\left(x + q_1 {z^{(1)}} + \ldots + q_{\gamma}{z^{(\gamma)}}\right) d{{^*}{\bar{\sigma}}}(x) \\
    &= \starint_{r \cdot S_{N, z^{(1)}, \ldots, z^{(\gamma)}} + q_1 {z^{(1)}} + \ldots + q_{\gamma}{z^{(\gamma)}}} {^*}f(x) d{{^*}{\bar{\sigma}}}(x) \\
    &= \starint_{S_{A_N}} {^*}f(x) d{{^*}{\bar{\sigma}}}(x).
\end{align*}

The first line follows from the fact that ${^*}h$ is bounded by a real number (and is hence $S$-integrable) and the following fact that is true for all $x \in {^*}\mathbb{R}^N$ (due to the nonstandard characterization of the uniform continuity of $f \colon \mathbb{R}^k \rightarrow \mathbb{R}$ and \eqref{approximation after truncation}): 
\begin{align*}
    {^*}h(x) &= {^*}f\left(x + p_1 (u^{(1)})_{(k)} + \ldots p_{\gamma} (u^{(\gamma)})_{(k)} \right) \\
    &\approx {^*}f\left(x + q_1 {(z^{(1)})_{(k)}} + \ldots + q_{\gamma}{(z^{(\gamma)})_{(k)}}\right) \\
    &= {^*}f\left(x + q_1 {z^{(1)}} + \ldots + q_{\gamma}{z^{(\gamma)}}\right).
\end{align*}
The second line follows by transfer of the translation properties of the uniform surface measures. The third line follows from \eqref{SLn}. 
\end{proof}

Using Theorem \ref{main}, we immediately deduce that the first $k$ coordinates of almost any point of $S_{A_N}$ are finite. 

\begin{theorem}\label{almost all points are finite}
     Almost all points of $S_{A_N}$ have finite projections to ${^*}\mathbb{R}^k$, i.e.,
     \begin{align*}
         L\bar{\sigma}(\{x \in S_{A_N}: x_1, \ldots, x_k \in {^*}\mathbb{R}_{\text{fin}}\}) = 1.
     \end{align*}
\end{theorem}

\begin{proof}
We prove this for $k = 1$ (the general case follows from the fact that the intersection of finitely many almost sure events is almost sure). For each $m \in \mathbb{N}$, consider the function $f_m$ that is equal to $1$ on $(-m+1, m-1)$, equal to zero on $\mathbb{R} \backslash (-m, m)$, and is linear in between. We thus have
\begin{align*}
    L\bar{\sigma}(x_1 \in {^*}(-m, m))  
    &= \mathbb{E}_{S_{A_N}}(\st(\mathbbm{1}_{{^*}(-m,m)})) 
    \\
    &= \mathbb{E}_{S_{A_N}}(\st({^*}\mathbbm{1}_{(-m,m)})) 
    \\
    &\geq \mathbb{E}_{S_{A_N}}(\st({^*}f_m)) \\
    &\geq \int_{\mathbb{R}^k} f
    _m d\mu. \hspace{4 cm} [\text{using Theorem } \ref{main}] 
\end{align*}
As a consequence, we obtain
\begin{gather*}
    1 \geq L\bar{\sigma}(x_1 \in {^*}\mathbb{R}_{\text{fin}}) = L\bar{\sigma}(\cup_{m \in \mathbb{N}} \{x_1 \in {^*}(-m, m)\}) = \lim_{m \rightarrow \infty} L\bar{\sigma}(x_1 \in {^*}(-m, m)) \\
     \Rightarrow 1 \geq L\bar{\sigma}(x_1 \in {^*}\mathbb{R}_{\text{fin}}) \geq \lim_{m \rightarrow \infty} \int_{\mathbb{R}^k} f
    _m d\mu = 1 \\
    \Rightarrow L\bar{\sigma}(x_1 \in {^*}\mathbb{R}_{\text{fin}}) = 1,
\end{gather*}
thus completing the proof.
\end{proof}

Using Theorem \ref{main} and Theorem \ref{almost all points are finite}, we are now able to generalize the limiting spherical integral result to all bounded continuous functions on $\mathbb{R}^k$.

\begin{theorem}\label{all continuous functions}
Let $f\colon \mathbb{R}^k \rightarrow \mathbb{R}$ be a bounded continuous function. Then
\begin{align*}
    \lim_{n \rightarrow \infty} \int_{S_{A_n}} f d{\bar{\sigma}} = \int_{\mathbb{R}^k} f d{\mu}_{\bar{\eta}, u^{(1)}, \ldots, u^{(\gamma)}}.
\end{align*}
\end{theorem}
\begin{proof}
Let $f \colon \mathbb{R}^k \rightarrow \mathbb{R}$ be bounded and continuous. For each $m \in \mathbb{N}$, let $f_m$ be the restriction of $f$ to $[-m, m]^k$, i.e., $f_m \defeq f \cdot \mathbbm{1}_{[-m,m]^k}$. Fix $N > \mathbb{N}$. Since $f$ is bounded, $\st({^*}f)$ is $S$-integrable. This shows:

\begin{align}\label{main result first}
    \starint_{S_{A_N}} {^*}f d\bar{\sigma} \approx \int_{S_{A_N}} \st({^*}f) dL\bar{\sigma}. 
\end{align}

Using Theorem \ref{almost all points are finite} and applying dominated convergence theorem, we obtain:
\begin{align}\label{main result second}
    \int_{S_{A_N}} \st({^*}f) dL\bar{\sigma} = \lim_{m \rightarrow \infty} \int_{S_{A_N}} \st({^*}f_m) dL\bar{\sigma}.
\end{align}

The right side of \eqref{main result second} equals $\lim_{m \rightarrow \infty} \int_{\mathbb{R}^k} f_m (x) d\mu_{\bar{\eta}, u^{(1)}, \ldots, u^{(\gamma)}}$ using Theorem \ref{main}. Thus dominated convergence theorem and \eqref{main result first} now completes the proof. 
\end{proof}
\end{section}

\begin{appendix}
\section{Some results on linear independence}
\renewcommand\thesection{A}
\begin{lemma}\label{independence}
    Let $u^{(1)}, \ldots, u^{(\gamma)}$ be $\mathbb{R}$-linearly independent vectors in $\mathbb{R}^{\mathbb{N}}$. For all $K > \mathbb{N}$, $(u^{(1)})_{(K)}$, $\ldots, (u^{(\gamma)})_{(K)}$ are ${^*}\mathbb{R}$-linearly independent. As a consequence, $(u^{(1)})_{(m)}, \ldots, (u^{(\gamma)})_{(m)}$ are $\mathbb{R}$-linearly independent for all large $m \in \mathbb{N}$.
    \end{lemma}
\begin{proof}
Fix $M > \mathbb{N}$. Suppose, if possible, that $(u^{(1)})_{(M)}$, $\ldots, (u^{(\gamma)})_{(M)}$ are not ${^*}\mathbb{R}$-linearly independent. Then there exist $a_1, \ldots, a_{\gamma} \in {^*}\mathbb{R}$ such that not all the $a_i$ are zero and 
\begin{align}\label{independence equality}
    \sum_{i = 1}^{\gamma}a_i (u^{(i)})_{(M)} = \mathbf{0} \in {^*}\mathbb{R}^M.
\end{align}

Let $a = \max_{i \in \{1, \ldots, \gamma\}}\abs{a_i}$. Since the $a_i$ are not all equal to zero, we have $a > 0$. For each $i \in \{1, \ldots, \gamma\}$, let $b_i = \frac{a_i}{a}$. Divide both sides of (\ref{independence equality}) by $a$ to get
\begin{align}\label{independence equality 2}
    \sum_{i = 1}^{\gamma}b_i (u^{(i)})_{(M)} = \mathbf{0} \in {^*}\mathbb{R}^M.
\end{align}

All the $b_i$ are bounded above in absolute value by $1$. Thus taking standard parts along the coordinates in $\mathbb{N}$ on both sides of (\ref{independence equality 2}) gives
\begin{align*}
    \sum_{i =1}^{\gamma} \st(b_i) u^{(i)} = \mathbf{0} \in \mathbb{R}^{\mathbb{N}}. 
\end{align*}

Since  $u^{(1)}, \ldots, u^{(\gamma)}$ are $\mathbb{R}$-linearly independent, it follows that $\st(b_i) = 0$ for all $i \in \{1, \ldots, \gamma\}$. But this contradicts the fact that $\abs{b_i} = 1$ for at least one $i$ (namely for that index $i$ which makes $\abs{a_i}$ maximum).

Hence $(u^{(1)})_{(M)}$, $\ldots, (u^{(\gamma)})_{(M)}$ are ${^*}\mathbb{R}$-linearly independent for all $M > \mathbb{N}$. By underflow, $(u^{(1)})_{(m)}$, $\ldots, (u^{(\gamma)})_{(m)}$ are ${^*}\mathbb{R}$-linearly independent for all large $m \in \mathbb{N}$. But if $m \in \mathbb{N}$, then $(u^{(1)})_{(m)}$, $\ldots, (u^{(\gamma)})_{(m)}$ are all vectors in $\mathbb{R}^m$, so that the ${^*}\mathbb{R}$-linear independence of these vectors implies that they are also $\mathbb{R}$-linearly independent by transfer.
\end{proof}

It is interesting to note that Lemma \ref{independence} is a special case of a more general result on linear independence in infinite-dimensional functional spaces, which we include below. Techniques of this nature have been used in the past in the works of Ross. See \cite[Theorem 3]{Ross_linear} for a related idea used to give a nonstandard proof of the Riesz Representation Theorem.

\begin{theorem}
Let $X$ be an infinite set and let $F(X, \mathbb{R})$ denote the vector space of functions from $X$ to $\mathbb{R}$. Suppose $f_1, \ldots, f_{\gamma}$ are linearly independent in $F(X, \mathbb{R})$. Then there is a number $m_0 \in \mathbb{N}$ such that for all $m \in \mathbb{N}_{> m_0}$, there are points $x_1, \ldots, x_m \in X$ for which the vectors $\left(f_i(x_1), \ldots, f_i(x_m)\right)_{i\in \{1, \ldots, \gamma\}}$ are linearly independent in $\mathbb{R}^m$. 
\end{theorem}

\begin{remark}
If $X$ is countable, say, $X = \mathbb{N}$ (in which case $F(X, \mathbb{R})$ is just the vector space of sequences of real numbers), then for any linearly independent vectors $u^{(1)}, \ldots, u^{(\gamma)}$, there is an $m_0 \in \mathbb{N}$ such that for all $m \in \mathbb{N}_{>m_0}$, the vectors $\left({u^{(i)}}_{1}, \ldots, {u^{(i)}}_m\right)_{i\in \{1, \ldots, \gamma\}}$ are linearly independent in $\mathbb{R}^m$. In this sense, Lemma \ref{independence} is a corollary of this theorem.
\end{remark}

\begin{proof}
For brevity, we will write $f_i(x_1, \ldots, x_m)$ to denote $f_i((x_1), \ldots, f_i(x_m))$. Let $F_0$ be a hyperfinite set such that $X \subseteq F_0 \subseteq {^*}X$. One obtains $F_0$ via saturation as an element of the following set:
$$\cap_{x \in X} \{F \in {^*}\mathcal{P}_{\text{fin}}(X): x \in F\}.$$

Suppose that the internal cardinality of $F_0$ is $\abs{F_0} = N$. Since $X$ is infinite, there is an injective map $s\colon \mathbb{N} \rightarrow X$. Extend this map to get a hyperfinite sequence $(x_i)_{i \in {^*}\mathbb{N}}$ of distinct elements in ${^*}X$ (by taking $x_i \defeq {^*}s(i)$). For each natural number $m$, we let $[m]$ denote the set $\{1, \ldots, m\}$. Also, for two subsets $A, B$ in the standard universe, we let $\text{Bij}(A, B)$ denote the set of bijections between $A$ and $B$ (so it is empty if $A$ and $B$ have different cardinalities). Consider the following internal set:
$$\mathcal{G} \defeq \left\{m \in {^*}\mathbb{N}: \exists \phi \in {^*}\text{Bij}([N], F_0) \text{ such that } {^*}s|_{[m]} = \phi|_{[m]}\right\}.$$

Clearly, $\mathcal{G}$ contains $\mathbb{N}$ and hence contains an $M > \mathbb{N}$ by overflow. Let $\phi$ be the bijection that witnesses the inclusion of $M$ in $\mathcal{G}$. Let $y_i = \phi(i) (= x_i)$ for all $i \in [M]$. Extend $\phi$ (by transfer of the fact that any injective map from an initial set of $\mathbb{N}$ to $X$ can be extended to an injective map from $\mathbb{N}$ to $X$) to an internal injective map $\Phi\colon {^*}\mathbb{N} \rightarrow {^*}X$, and still call $y_i = \Phi(i)$ for all $i \in {^*}\mathbb{N}$. Recall that $y_i = x_i$ for all $i \in [M]$ (in particular for all $i \in \mathbb{N}$).

We claim that $\left({^*} f_i(y_1), \ldots, {^*}f_i(y_N)\right)_{i \in [\gamma]}$ are ${^*}\mathbb{R}$-linearly independent in ${^*}\mathbb{R}^N$. For if not, then there exist $a_1, \ldots, a_{\gamma} \in {^*}\mathbb{R}$, not all zero, such that $\sum_{i \in [\gamma]} a_i {^*}f_i |_{F_0} = 0$. As in the proof of Lemma \ref{independence}, we divide both sides by $\max\{\abs{a_1}, \ldots, \abs{a_{\gamma}}\}$ and restrict the functions to $X$ to get a contradiction to the linear independence of $f_1, \ldots, f_{\gamma}$. 

Since $N > \mathbb{N}$ was arbitrary, the following internal set contains ${^*}\mathbb{N} \backslash \mathbb{N}$:
$$\{n \in {^*}\mathbb{N}: ~\exists A \in {^*}\mathcal{P}_{\text{fin}}(X) [(\abs{A} = n) \land (f_1 |_A, \ldots, f_{\gamma}|_A ~ {^*}\mathbb{R}\text{-linearly independent)}]\}.$$

By underflow, there is an $m_0 \in \mathbb{N}$ in this set. For any $m \in \mathbb{N}_0$, the transfer of the following sentence completes the proof of the first part of this theorem:
$$\exists A \in {^*}\mathcal{P}_{\text{fin}}(X) [(\abs{A} = n) \land (f_1 |_A, \ldots, f_{\gamma}|_A \text{ are } {^*}\mathbb{R}\text{-linearly independent)}].$$

For the second part of the theorem, replace $A$ by $[n]$ in the above underflow argument and then proceed as before. 
\end{proof}

\renewcommand\thesection{B}
\section{Working with infinitesimally separated linear spaces}
In an internal inner product space $V$ (over ${^*}\mathbb{R}$ or ${^*}\mathbb{C}$), a collection of vectors $\mathcal{V}$ is said to satisfy the \textit{separation property (SP)} if the following holds:
\begin{align}\label{easier separation property}
    \text{For any } v \in \mathcal{V}, ~  \norm{v - P_{\text{span}{(\mathcal{V}\backslash \{v\})}} (v)} \not\approx 0.
\end{align}

Here, for a subspace $H$, the vector $P_H(v)$ denotes the orthogonal projection of the vector $v$ onto $H$. The following equivalent version of SP is more convenient for our applications (the equivalence follows from the linear algebraic fact that distance of a vector from its projection onto a larger subspace cannot be bigger than the distance from its projection onto a smaller subspace). 

\begin{align}\label{defining separation property}
    \text{For any } v \in \mathcal{V} \text{ and any subcollection } \mathcal{V}' \subseteq \mathcal{V} \backslash \{v\} : ~  \norm{v - P_{\text{span}{(\mathcal{V}')}} (v)} \not\approx 0.
\end{align}

When working with spheres intersected by hyperplanes, we often need to orthonormalize different sets of linearly independent vectors (corresponding to two different hyperplanes). If two such sets of vectors can be matched with each other in the sense that any pair is only infinitesimally apart, then we can make such a matching with their orthonormalizations as well, provided the original set of vectors satisfies the Separation Property (this is proved in Theorem \ref{Rotation Claim duplicate}). 

We first prove a preliminary result that shows that any collection of vectors satisfying SP must be linearly independent. Note that the converse is not true---one could take vectors $\{e_1, \epsilon e_2\}$, or $\{e_1, e_1 + \epsilon e_2\}$ in ${^*}\mathbb{R}^2$, where $\epsilon$ is an infinitesimal. In what follows, we call a vector $v$ \textit{infinitesimal} if $\norm{v} \approx 0$.

\begin{proposition}\label{SP implies LI}
Suppose a collection of vectors $\mathcal{V}$ satisfies SP. Then $\mathcal{V}$ does not contain any infinitesimal. Furthermore, $\mathcal{V}$ is ${^*}\mathbb{R}$-linearly independent.  
\end{proposition}

\begin{proof}
The first part follows from the fact that any orthogonal projection operator has norm at most $1$. Indeed, for any $v \in \mathcal{V}$,
$$\norm{v - P_{\text{span}(\mathcal{V} \backslash \{v\})}(v)} \leq \norm{v} + \norm{P_{\text{span}(\mathcal{V} \backslash \{v\})}(v)} \leq 2\norm{v},$$
which would be infinitesimal if $v$ is an infinitesimal vector. 

Now if $\mathcal{V}$ were not linearly independent, then there would exist a vector $v \in \mathcal{V}$ which could be written as a linear combination of vectors from some subcollection $\mathcal{V}' \subseteq V \backslash \{v\}$. But then we would have $P_{\text{span}(\mathcal{V}')}(v) = v$, violating the Separation Property, since we have already shown $v$ to be non-infinitesimal.
\end{proof}

\begin{theorem}\label{Rotation Claim duplicate} Let $V$ be an internal inner product space. Let $\gamma \in \mathbb{N}$. For each $i \in \{1, \ldots, \gamma\}$, let $v^{(i)}, v'^{(i)} \in V$ be such that the following conditions hold: 
\begin{enumerate}[(i)]
    \item\label{rotation 1}  The collections $\{v^{(1)}, \ldots, v^{(\gamma)}\}$ and $\{v'^{(1)}, \ldots, v'^{(\gamma)}\}$ both satisfy the Separation Property. .  
      \item\label{rotation 2} $\norm{v^{(i)}}, \norm{v'^{(i)}} \in {^*}\mathbb{R}_{\text{fin}}.$
    \item\label{rotation 3} $\norm{v^{(i)} - v'^{(i)}} \approx 0.$
\end{enumerate}
Then there exist orthonormal sets $\{w^{(1)}, \ldots, w^{(\gamma)}\}$ and $\{z^{(1)}, \ldots, z^{(\gamma)}\}$ with the following properties:
\begin{enumerate}
    \item \label{SP1} For any $i \in \{1, \ldots, \gamma\}$, we have 
    \begin{align*}
    \text{span}(v^{(1)}, \ldots, v^{(i)}) &= \text{span}(w^{(1)}, \ldots, w^{(i)}), \\
    \text{ and } \text{span}(v'^{(1)}, \ldots, v'^{(i)}) &= \text{span}(z^{(1)}, \ldots, z^{(i)}).
    \end{align*}
    
    \item\label{SP2} For all $i \in \{1, \ldots, \gamma\}$, we have $\norm{w^{(i)} - z^{(i)}} \approx 0$.
\end{enumerate}
\end{theorem}

\begin{proof}
Use the Gram-Schmidt algorithm on $\{v^{(1)}, \ldots, v^{(\gamma)}\}$ and $\{v'^{(1)}, \ldots, v'^{(\gamma)}\}$ to obtain $\{w^{(1)}, \ldots, w^{(\gamma)}\}$ and $\{z^{(1)}, \ldots, z^{(\gamma)}\}$ respectively. We thus have:
\begin{align} \label{Gram-Schmidt 1}
           w^{(1)} &\defeq \frac{v^{(1)}}{\norm{v^{(1)}}}, \\   
         \nonumber  w^{(2)} &\defeq \frac{v^{(2)} - {\langle v^{(2)}, w^{(1)}\rangle}w^{(1)}}{\norm{v^{(2)} - {\langle v^{(2)}, w^{(1)}\rangle}w^{(1)}}}, \\
         \nonumber w^{(3)} &\defeq \frac{v^{(3)} - {\langle v^{(3)}, w^{(1)}\rangle}w^{(1)} - {\langle v^{(3)}, w^{(2)}\rangle}w^{(2)}}{\norm{v^{(3)} - {\langle v^{(3)}, w^{(1)}\rangle}w^{(1)} - {\langle v^{(3)}, w^{(2)}\rangle}w^{(2)}}} ,\\
         \nonumber  &\vdots \\
         \nonumber  w^{(\gamma)} &\defeq \frac{v^{(\gamma)} - {\langle v^{(\gamma)}, w^{(1)}\rangle}w^{(1)} - \ldots - {\langle v^{(\gamma)}, w^{(\gamma - 1)}\rangle}w^{(\gamma - 1)}}{\norm{v^{(\gamma)} - {\langle v^{(\gamma)}, w^{(1)}\rangle}w^{(1)} - \ldots - {\langle v^{(\gamma)}, w^{(\gamma - 1)}\rangle}w^{(\gamma - 1)}}}, 
\end{align}
and
\begin{align} \label{Gram-Schmidt 2}
         z^{(1)} &\defeq \frac{v'^{(1)}}{\norm{v'^{(1)}}}, \\   
         \nonumber  z^{(2)} &\defeq \frac{v'^{(2)} - {\langle v'^{(2)}, z^{(1)}\rangle}z^{(1)}}{\norm{v'^{(2)} - {\langle v'^{(2)}, z^{(1)}\rangle}z^{(1)}}}, \\
         \nonumber z^{(3)} &\defeq \frac{v'^{(3)} - {\langle v'^{(3)}, z^{(1)}\rangle}z^{(1)} - {\langle v'^{(3)}, z^{(2)}\rangle}z^{(2)}}{\norm{v'^{(3)} - {\langle v'^{(3)}, z^{(1)}\rangle}z^{(1)} - {\langle v'^{(3)}, z^{(2)}\rangle}z^{(2)}}} ,\\
         \nonumber  &\vdots \\
         \nonumber  z^{(\gamma)} &\defeq \frac{v'^{(\gamma)} - {\langle v'^{(\gamma)}, z^{(1)}\rangle}z^{(1)} - \ldots - {\langle v'^{(\gamma)}, z^{(\gamma - 1)}\rangle}z^{(\gamma - 1)}}{\norm{v'^{(\gamma)} - {\langle v'^{(\gamma)}, z^{(1)}\rangle}z^{(1)} - \ldots - {\langle v'^{(\gamma)}, z^{(\gamma - 1)}\rangle}z^{(\gamma - 1)}}}.
\end{align}

These sets  $\{w^{(1)}, \ldots, w^{(\gamma)}\}$ and $\{z^{(1)}, \ldots, z^{(\gamma)}\}$ of internally orthonormal vectors satisfy \eqref{SP1} by construction. Therefore we need to only verify \eqref{SP2}. The proof of \eqref{SP2} will be done by induction on $i$. Observe that for $i = 1$, we have:
\begin{align}
    \norm{w^{(1)} - z^{(1)}} &= \norm{\frac{{\norm{v'^{(1)}}v^{(1)} - \norm{v^{(1)}}v'^{(1)}}}{\norm{v^{(1)}}\cdot\norm{v'^{(1)}}}} \nonumber \\
    &= \frac{\norm{v^{(1)} - \frac{\norm{v^{(1)}}}{\norm{v'^{(1)}}}v'^{(1)}}}{\norm{v^{(1)}}} \nonumber \\
    &\approx \frac{\norm{v^{(1)} - 1 \cdot v'^{(1)}}}{\st\left(\norm{v^{(1)}}\right)} \nonumber\\
    &\approx 0 \label{base of induction}.
\end{align}

Similarly, 
\begin{align}
      \norm{w^{(2)} - z^{(2)}} &=  \norm{ \frac{v^{(2)} - {\langle v^{(2)}, w^{(1)}\rangle}w^{(1)}}{\norm{v^{(2)} - {\langle v^{(2)}, w^{(1)}\rangle}w^{(1)}}} -  \frac{v'^{(2)} - {\langle v'^{(2)}, z^{(1)}\rangle}z^{(1)}}{\norm{v'^{(2)} - {\langle v'^{(2)}, z^{(1)}\rangle}z^{(1)}}}} \nonumber\\
      &= \frac{\norm{\alpha v^{(2)} - \alpha \inp*{v^{(2)}}{w^{(1)}}w^{(1)} - \beta v'^{(2)} + \beta \inp*{v'^{(2)}}{z^{(1)}}z^{(1)}}}{\alpha \beta} \label{second step},
\end{align}
where $\alpha = \norm{v'^{(2)} - {\langle v'^{(2)}, z^{(1)}\rangle}z^{(1)}}$ and $\beta = \norm{v^{(2)} - {\langle v^{(2)}, w^{(1)}\rangle}w^{(1)}}$. 

Geometrically, $\alpha$ (respectively $\beta$) represents the orthogonal projection of $v^{(2)}$ (respectively $v'^{(2)}$) onto the span of $v^{(1)}$ (respectively $v'^{(1)}$). Hence, by the SP condition, it follows that $\alpha\beta$ is non-infinitesimal. 

By repeated uses of triangle inequality and Cauchy-Schwarz inequality, we have:
\begin{align}
    &\abs{\alpha - \beta} \nonumber \\
    \leq &\norm{v^{(2)} - v'^{(2)}} +  \norm{\langle v^{(2)}, w^{(1)}\rangle w^{(1)} - \langle v'^{(2)}, z^{(1)}\rangle z^{(1)}} \nonumber \\
    = &\norm{v^{(2)} - v'^{(2)}} +  \norm{\langle v^{(2)}, w^{(1)} - z^{(1)}\rangle w^{(1)} + \inp{v^{(2)}}{z^{(1)}} w^{(1)} -  \langle v'^{(2)}, z^{(1)}\rangle z^{(1)}} \nonumber \\
    \leq &\norm{v^{(2)} - v'^{(2)}} + \abs{\langle v^{(2)}, w^{(1)} - z^{(1)}\rangle} \norm{w^{(1)}} \nonumber \\ 
    &+ \norm{\inp{v^{(2)}}{z^{(1)}} w^{(1)} -  \langle v'^{(2)}, z^{(1)}\rangle z^{(1)}} \nonumber \\
    \leq &\norm{v^{(2)} - v'^{(2)}} + \abs{\langle v^{(2)}, w^{(1)} - z^{(1)}\rangle} \norm{w^{(1)}} \nonumber \\
     ~ \hspace{0.4cm}&+ \norm{\inp{v^{(2)}}{z^{(1)}} \left(w^{(1)} - z^{(1)} \right) + \inp{v^{(2)} - v'^{(2)}}{z^{(1)}} z^{(1)}} \nonumber \\
\leq &\norm{v^{(2)} - v'^{(2)}} + \norm{v^{(2)}} \norm{w^{(1)} - z^{(1)}} \norm{w^{(1)}} + \norm{v^{(2)}}\norm{z^{(1)}} \norm{w^{(1)} - z^{(1)}} \nonumber \\
&+ \norm{v^{(2)} - v'^{(2)}}\norm{z^{(1)}} \norm{z^{(1)}} \nonumber,
\end{align}
which is infinitesimal by the hypothesis. Hence, we have 
\begin{align}
\abs{\alpha - \beta} \approx 0. \label{alpha - beta}
\end{align}

Using triangle inequality and Cauchy-Schwarz inequality a few times, we have:
\begin{align}
    \abs{\langle{v^{(2)}, w^{(1)}}\rangle - \langle{v'^{(2)}, z^{(1)}}\rangle} \nonumber
    &= \abs{\inp{v^{(2)} - v'^{(2)}}{w^{(1)}} + \inp{v'^{(2)}}{w^{(1)} - z^{(1)}}} \\
    &\leq \abs{\inp{v^{(2)} - v'^{(2)}}{w^{(1)}}} + \abs{\inp{v'^{(2)}}{w^{(1)} - z^{(1)}}} \nonumber \\
    &\leq \norm{v^{(2)} - v'^{(2)}}\norm{w^{(1)}} + \norm{v'^{(2)}} \norm{w^{(1)} - z^{(1)}}
     \label{super long expression}
\end{align}

The right side of \eqref{super long expression} is an infinitesimal by the hypothesis and \eqref{base of induction}. Since $\langle{v^{(2)}, w^{(1)}}\rangle, \langle{v'^{(2)}, z^{(1)}}\rangle$ are in ${^*}\mathbb{R}_{\text{fin}}$ (one can see this using Cauchy-Schwarz inequality), we thus get:
\begin{align}
    \langle{v^{(2)}, w^{(1)}}\rangle \approx \langle{v'^{(2)}, z^{(1)}}\rangle \label{approximation in super long expression}
\end{align}

Note that ${\alpha}, {\beta} \in {^*}\mathbb{R}_{\text{fin}}$ (one can see this by applying the triangle inequality and Cauchy-Schwarz inequality to the expressions for $\alpha$ and $\beta$). Using \eqref{alpha - beta} and \eqref{approximation in super long expression} in \eqref{second step} (and using the fact that $\st\colon {^*}\mathbb{R}_{\text{fin}} \rightarrow \mathbb{R}$ is a ring homomorphism), we get
\begin{align}\label{second step completed}
    \norm{w^{(2)} - z^{(2)}} \approx 0.
\end{align}

The proof of the case $i =2$ from the case $i = 1$ clearly generalizes to show, by induction, that $\norm{w^{(i)} - z^{(i)}} \approx 0$ for all $i \in \{1, \ldots, \gamma\}$.
\end{proof}

\begin{remark}
Theorem \ref{Rotation Claim duplicate} shows that if two internal subspaces have bases of finite vectors satisfying SP such that they can be matched in pairs of infinitesimal distances, then the same is true for the orthonormalizations of these bases as well. This allows one to ``rotate'' one subspace to another through an orthogonal transformation of infinitesimal norm, as done in Section 3.
\end{remark}

In all applications in this paper, the inner product space $V$ is taken to be ${^*}\mathbb{R}^N$ for some $N \in {^*}\mathbb{N}$ (usually taken to be hyperfinite). The vectors are usually hyperfinite truncations of an orthonormal collection of elements of $\ell^2(\mathbb{R})$. We next show that Theorem \ref{Rotation Claim duplicate} is applicable in that setting. 

\begin{proposition}\label{hyperfinite truncations are SP}
Let $\{u^{(1)}, \ldots, u^{(\gamma)}\}$ be a finite collection of orthonormal vectors in $\ell^2(\mathbb{R})$. 
\begin{enumerate}
    \item\label{hyperfinite Sp1} For any $N > \mathbb{N}$, the collection $\{(u^{(1)})_{(N)}, \ldots, (u^{(\gamma)})_{(N)}\}$ satisfies the Separation Property. 
    \item\label{hyperfinite SP2} For any $N > M > \mathbb{N}$, the collection of vectors $\{(u^{(1)})_{(M)}, \ldots, (u^{(\gamma)})_{(M)}\}$ (canonically viewed as vectors in ${^*}\mathbb{R}^N$) and $\{(u^{(1)})_{(N)}, \ldots, (u^{(\gamma)})_{(N)}\}$ satisfy the conditions in Theorem \ref{Rotation Claim duplicate}. 
\end{enumerate}
\end{proposition}
\begin{proof}
Let $\{u^{(1)}, \ldots, u^{(\gamma)}\}$ be as in the statement of the proposition. Let $N > \mathbb{N}$. By Lemma \ref{independence}, $\{(u^{(1)})_{(N)}, \ldots, (u^{(\gamma)})_{(N)}\}$ is linearly independent. Therefore, we can apply the Gram-Schmidt orthonormalization to obtain the corresponding orthonormal set $\{z^{(1)}, \ldots, z^{(\gamma)}\}$. 

Take $i \in \{1, \ldots, \gamma\}$, and let $\mathcal{V}' \defeq \{(u^{(j_1)})_{(N)}, \ldots, (u^{(j_t)})_{(N)}\}$ be a subcollection not containing $(u^{(i)})_{(N)}$. Then we have:
\begin{align*}
    \norm{(u^{(i)})_{(N)} - P_{\text{span}(\mathcal{V}')}((u^{(i)})_{(N)})} &= \norm{(u^{(i)})_{(N)} - \sum_{\theta = 1}^t \inp{(u^{(i)})_{(N)}}{z^{(j_{\theta})}}z^{(j_{\theta})}} \\
    &\geq \norm{(u^{(i)})_{(N)}} - \norm{\sum_{\theta = 1}^t {\inp{(u^{(i)})_{(N)}}{z^{(j_{\theta})}}} {z^{(j_{\theta})}}} \\
    &\geq \norm{(u^{(i)})_{(N)}} - \sum_{\theta = 1}^t \abs{\inp{(u^{(i)})_{(N)}}{z^{(j_{\theta})}}}\norm{z^{(j_{\theta})}} \\
    &= \norm{(u^{(i)})_{(N)}} - \sum_{\theta = 1}^t \abs{\inp{(u^{(i)})_{(N)}}{z^{(j_{\theta})}}}
\end{align*}

The second and third lines follow by triangle inequality, and the fourth line follows from the fact that $\norm{z^{(j)}} = 1$ for all $j \in \{1, \ldots, \gamma\}$. Thus, to prove \ref{hyperfinite Sp1}, it suffices to show the following claim:

\begin{claim}\label{appendix claim}
We have $\inp {(u^{(i)})_{(N)}} {z^{(j)}} \approx 0$ for all $j \in \{1, \ldots, \gamma\} \backslash \{i\}$.
\end{claim}

This is a straightforward consequence of the precise formulae for $z^{(i)}$ as per the Gram-Schmidt orthonormalization procedure (see below):
\begin{align} \label{Gram-Schmidt appendix}
           z^{(1)} &\defeq \frac{(u^{(1)})_{(N)}}{\norm{(u^{(1)})_{(N)}}}, \\   
         \nonumber  z^{(2)} &\defeq \frac{(u^{(2)})_{(N)} - {\langle (u^{(2)})_{(N)}, z^{(1)}\rangle}z^{(1)}}{\norm{(u^{(2)})_{(N)} - {\langle (u^{(2)})_{(N)}, z^{(1)}\rangle}z^{(1)}}}, \\
         \nonumber z^{(3)} &\defeq \frac{(u^{(3)})_{(N)} - {\langle (u^{(3)})_{(N)}, z^{(1)}\rangle}z^{(1)} - {\langle (u^{(3)})_{(N)}, z^{(2)}\rangle}z^{(2)}}{\norm{(u^{(3)})_{(N)} - {\langle (u^{(3)})_{(N)}, z^{(1)}\rangle}z^{(1)} - {\langle (u^{(3)})_{(N)}, z^{(2)}\rangle}z^{(2)}}} ,\\
         \nonumber  &\vdots \\
         \nonumber  z^{(\gamma)} &\defeq \frac{(u^{(\gamma)})_{(N)} - {\langle (u^{(\gamma)})_{(N)}, z^{(1)}\rangle}z^{(1)} - \ldots - {\langle (u^{(\gamma)})_{(N)}, z^{(\gamma - 1)}\rangle}z^{(\gamma - 1)}}{\norm{(u^{(\gamma)})_{(N)} - {\langle (u^{(\gamma)})_{(N)}, z^{(1)}\rangle}z^{(1)} - \ldots - {\langle (u^{(\gamma)})_{(N)}, z^{(\gamma - 1)}\rangle}z^{(\gamma - 1)}}}.
\end{align}

Indeed, the fact that $\inp{u^{(i)}}{u^{(j)}}_{\ell^2(\mathbb{R})} = \lim_{n \rightarrow \infty} \inp{(u^{(i)})_{(n)}}{(u^{(j)})_{(n)}} = \delta_{ij}$ implies (by the nonstandard characterization of limits) that $\inp{(u^{(i)})_{(N)}}{(u^{(j)})_{(N)}} \approx 0$ for $i \neq j$, which proves the claim. This completes the proof of \eqref{hyperfinite Sp1}. 

Now, let $N > M > \mathbb{N}$. By \eqref{hyperfinite Sp1}, both $\{(u^{(1)})_{(M)}, \ldots, (u^{(\gamma)})_{(M)}\}$ (viewed canonically as vectors in ${^*}\mathbb{R}^N$ and $\{(u^{(1)})_{(N)}, \ldots, (u^{(\gamma)})_{(N)}\}$ satisfy SP. Also, by a similar argument as in the proof of Claim \ref{appendix claim}, we have
\begin{align*}
    \norm{(u^{(i)})_{(N)}} \approx \norm{(u^{(i)})_{(M)}} &\approx 1 \text{ for all } i \in \{1, \ldots, \gamma\}, \\
    \text{and } \norm{(u^{(i)})_{(N)}- (u^{(i)})_{(M)}} 
    &\approx 0 \text{ for all } i \in \{1, \ldots, \gamma\}.
\end{align*}
This completes the proof of \eqref{hyperfinite SP2}.
\end{proof}
\end{appendix}

\par\bigskip\noindent
{\bf Acknowledgment.} The author thanks Ambar Sengupta for introducing this problem and Karl Mahlburg for helpful comments on the paper. The author thanks David Ross for discussions on nonstandard linear algebra. The author is grateful to the online resources of the TikZ community, more specifically to a reccent paper by Peterson and Sengupta \cite{Sengupta-Peterson} for source codes of some figures. The author thanks Kristopher Hollingsworth for help with editing the figures. 

\bibliography{References}
\bibliographystyle{amsplain}
\end{document}